\documentclass[11pt,a4paper, reqno]{amsart}
 \usepackage{amsaddr}
\usepackage[a4paper, left = 2cm,  hmargin={25mm,25mm},vmargin={25mm,25mm}]{geometry}
\usepackage[utf8]{inputenc}
\usepackage{geometry}
\usepackage[english]{babel}
\usepackage{graphicx}
\usepackage{color}
\usepackage{xcolor}
\usepackage{pdfpages}
\usepackage{amsbsy}
\usepackage{amssymb}
\usepackage{amsmath}
\usepackage{bm}
\usepackage{mathtools}
\usepackage{matlab-prettifier}
\usepackage{amsthm}
\usepackage{empheq}
\usepackage{tikz-cd}
\usepackage{amsfonts}
\usepackage{array}
\usepackage{hyperref}
\usepackage[OT2,T1]{fontenc}
\usepackage{csquotes}
\usepackage{mathabx}
\usepackage{bm}
\usepackage{bigints}

\newcommand\m[1]{\begin{pmatrix}#1\end{pmatrix}} 

\DeclareSymbolFont{cyrletters}{OT2}{wncyr}{m}{n}
\DeclareMathSymbol{\Sha}{\mathalpha}{cyrletters}{"58}
\newcolumntype{P}[1]{>{\centering\arraybackslash}p{#1}}
\everymath{\displaystyle}
\newtheorem{theorem}{Theorem}[section]

\newtheorem{lemma}[theorem]{Lemma}

\newtheorem{defn}[theorem]{Definition}

\newtheorem{proposition}[theorem]{Proposition}

\newtheorem{remark}[theorem]{Remark}

\newtheorem*{notation}{Notation}
\newtheorem{corollary}[theorem]{Corollary}

\definecolor{lightgray}{gray}{0.5}
\newtheorem{innercustomprop}{Proposition}
\newenvironment{myprop}[1]
  {\renewcommand\theinnercustomprop{#1}\innercustomprop}
  {\endinnercustomprop}
\newtheorem{innercustomthm}{Theorem}
\newenvironment{mythm}[1]
  {\renewcommand\theinnercustomthm{#1}\innercustomthm}
  {\endinnercustomthm}

\usepackage[backend=biber]{biblatex}
\begin{document}
\title{Analytic Properties of an orthogonal Fourier-Jacobi Dirichlet Series}
\vspace{-0.5cm}
\author{Rafail Psyroukis}
\address{Department of Mathematical Sciences, Durham University, 
South. Rd.\\ Durham, DH1 3LE, U.K..}
\email{rafail.psyroukis@durham.ac.uk}
\maketitle
\vspace{-1cm}
\begin{abstract}
    We investigate the analytic properties of a Dirichlet series involving the Fourier-Jacobi coefficients of two cusp forms for orthogonal groups of signature $(2,n+2)$. Using an orthogonal Eisenstein series of Klingen type, we obtain an integral representation for this Dirichlet series. In the case when the corresponding lattice has only one $1$-dimensional cusp, we rewrite this Eisenstein series in the form of an Epstein zeta function. If additionally $4 \mid n$, we deduce a theta correspondence between this Eisenstein series and a Siegel Eisenstein series for the symplectic group of degree $2$. We obtain, in this way, the meromorphic continuation of the Dirichlet series to $\mathbb{C}$ as a corollary. In the case of the $E_8$ lattice, we are able to further deduce a precise functional equation for the Dirichlet series.
\end{abstract}
\vspace{-0.5cm}
\renewcommand{\abstractname}{Acknowledgements}
\begin{abstract}
This work was supported by the Additional Funding Programme for Mathematical Sciences, delivered by EPSRC (EP/V521917/1) and the Heilbronn Institute for Mathematical Research, as well as by a scholarship from the Onassis Foundation. The author would also like to thank Prof. Thanasis Bouganis for his continuous guidance and support.
\end{abstract}
\section{Introduction}
Kohnen and Skoruppa in \cite{kohnen_skoruppa} considered the Dirichlet series
\begin{equation}\label{kohnen-skoruppa}
    D_{F,G}(s) = \zeta(2s-2k+4)\sum_{N=1}^{\infty} \langle \phi_N, \psi_N \rangle N^{-s},
\end{equation}
where $\phi_N, \psi_N$ are the $N$th Fourier-Jacobi coefficients of two degree two Siegel cusp forms $F,G$ of integral weight $k\geq0$ and $\langle \;, \;\rangle$ denotes the Petersson inner product on the space of Jacobi forms of index $N$ and weight $k$. This can be formally seen as the analogue of the Rankin-Selberg convolution series for elliptic modular forms. They prove two main results: Firstly, this series admits a meromorphic continuation to $\mathbb{C}$ and satisfies a functional equation. Secondly, when $F$ is a Hecke eigenform and $G$ is in the Maass space, this series is proportional to the spinor $L$-function attached to $F$.\\\\
We wish to analyse their first result in greater detail. The method for their proof can be summarised in two steps: the first (and easier) is to obtain an integral representation for $D_{F,G}$ using a non-holomorphic Eisenstein series of Klingen type, and the second is to prove the meromorphic continuation and functional equation for this Eisenstein series. The proof of that involves writing the Eisenstein series in the form of an Epstein zeta function and then proving that it is a Mellin transform of a specific theta series.\\\\
This general method of proof has been successful in a number of other cases as well. For example, Raghavan and Sengupta in \cite{raghavan} and Gritsenko in \cite{gritsenko} considered the same problem, but in the case when $F,G$ are Hermitian cusp forms of degree $2$ over $\mathbb{Q}(i)$. In both papers, the authors managed to deduce the analytic properties of (a slightly differently normalised) $D_{F,G}$ by applying a very similar idea; however, there are two key differences in the second step above. The first one is that the Klingen Eisenstein series arises as the inner product of a theta series and a classical Eisenstein series for $\textup{SL}_2$. The second and more important one is that it is now necessary to apply some differential operators to the theta series first. The reason for this is that there are terms that cause the inner-product integral to diverge, so we need to eliminate them with the use of differential operators.\\\\
It should be noted here that the degree of the modular objects considered is not important. Yamazaki in \cite{yamazaki} generalised the result of Kohnen and Skoruppa for Siegel cusp forms of arbitrary degree $n \geq 1$. However, his method for deducing the analytic properties of the Eisenstein series of Klingen type is derived from the general Langlands' theory. Krieg, on the other hand, in \cite{krieg1991}, used theta correspondence in the arbitrary degree $n$ case for Siegel, Hermitian (over $\mathbb{Q}(i)$) and quaternionic (over the Hamiltonian quaternions) cusp forms. The use of differential operators was again essential for the last two cases. What is even more interesting is that Deitmar and Krieg in \cite{deitmar_krieg} managed to prove theta correspondence of Eisenstein series for the groups $\textup{Sp}_n(\mathbb{Z})$ and $\textup{O}(m,m)$, for arbitrary $m,n \geq 1$. Their proof is based on proving the existence of an invariant differential operator $R$, which, when applied to the suitable theta series, eliminates the terms that cause the divergence of the inner-product integral. This was a major advance because, up until then, all operators had to be found explicitly, and this could be done in only a handful of cases.\\\\
In this paper, we consider a non-normalised Dirichlet series $\mathcal{D}_{F,G}$ of the same type as \eqref{kohnen-skoruppa} for two cusp forms $F,G$ of orthogonal groups of real signature $(2,n+2), \textup{ }n\geq 1$. Our purpose is to deduce its analytic continuation and functional equation in this case. The question of whether this series can be related to $L$-functions has been addressed in our previous paper \cite{psyroukis_orthogonal}. Our motivation is the accidental isogenies of orthogonal groups of signature $(2,n+2)$ with classical groups for small $n$ (e.g. $\textup{SO}(2,3)$ is isogenous to $\textup{Sp}_2$ and $\textup{SO}(2,4)$ is isogenous to $\textup{U}(2,2)$). Therefore, the results of Kohnen and Skoruppa in \cite{kohnen_skoruppa} and of Gritsenko in \cite{gritsenko} already indicate a connection in the $n=1,2$ case. In principle, we use the same general method of proof as all the aforementioned papers (except Yamazaki in \cite{yamazaki}). In our situation, we essentially obtain a theta-correspondence between Eisenstein series for $\textup{Sp}_2$ and $\textup{SO}(2,n+2)$. \\\\
Let us now give a brief description of our results. The setup is as follows: We start with an even positive definite symmetric 
matrix $S$ of rank $n \geq 1$ so that the lattice $L:=\mathbb{Z}^{n}$ is $\mathbb{Z}$-maximal. We then set
\begin{equation*}
    S_0 := \m{&&1\\ &-S&\\ 1&&}, \textup{ } S_1 := \m{&&1\\ &S_0&\\ 1&&}
\end{equation*}
of real signature $(1,n+1)$ and $(2,n+2)$ respectively. If now $K$ is a field containing $\mathbb{Q}$, we define the corresponding special orthogonal groups of $K$-rational points via
\begin{equation*}
    G^{*}_{K} := \{g \in \textup{SL}_{n+2}(K) \mid g^{t}S_0g = S_0\}, \textup{ }G_{K} := \{g \in \textup{SL}_{n+4}(K) \mid g^{t}S_1g = S_1\}.
\end{equation*}
It turns out that $G_{\mathbb{R}}^{0}$, the connected component of the identity, acts on a suitable tube domain, which we will call $\mathcal{H}_{S} \subset \mathbb{C}^{n+2}$, and an appropriate factor of automorphy $j(\gamma, Z)$ can be defined. Let us now denote by $\Gamma_{S}$ the intersection of $G_{\mathbb{R}}^{0}$ with the stabiliser of the lattice $L_1 := \mathbb{Z}^{n+4}$ and let $F : \mathcal{H}_S \longrightarrow \mathbb{C}$ be an orthogonal cusp form of weight $k$ with respect to $\Gamma_{S}$ (see Definitions \ref{modular_form_defn} and \ref{fourier coefficients}).\\\\
If $Z = (\omega, z, \tau) \in \mathcal{H}_{S}$, with $\omega, \tau \in \mathbb{H}$ and $z \in \mathbb{C}^{n}$, we can write
\begin{equation*}
    F(Z) = \sum_{m \geq 1}\phi_m(\tau, z)e^{2\pi i m \omega},
\end{equation*}
and we call the functions $\phi_m(\tau, z)$ the Fourier-Jacobi coefficients of $F$. Now, if $G$ is another such orthogonal cusp form with corresponding Fourier-Jacobi coefficients $\psi_{m}$, the object of interest is
\begin{equation*}
    \mathcal{D}_{F,G}(s) := \sum_{m\geq 1}\langle \phi_{m},\psi_{m}\rangle m^{-s},
\end{equation*}
where $\langle \;,\; \rangle$ is a suitable inner product defined on the space of Fourier-Jacobi forms of certain weight and (lattice) index (see Definition \ref{inner product_fourier-jacobi}). This series converges for $\textup{Re}(s)>k+1$ (see Lemma \ref{convergence_dirichlet}).\\\\
In order to apply the method we mentioned above, for $W \in \mathcal{H}_S$ and $\textup{Re}(s)>n+1$, let $E(W,s)$ denote an Eisenstein series of Klingen type in this setting (see Definition \ref{eisentein_defn}). The first step, namely the integral representation of the Dirichlet series, is then Proposition \ref{integral_representation_dirichlet}.
\begin{myprop}{4.4}
    For $W \in \mathcal{H}_S$ and $s \in \mathbb{C}$ with $\textup{Re}(s) > n+2$, we have
    \begin{equation*}
        \langle F(W)E(W,s), G(W)\rangle = \frac{1}{\#\textup{SO}(S;\mathbb{Z})}(4\pi)^{-(s+k-n-1)} \Gamma(s+k-n-1) \mathcal{D}_{F,G}(s+k-n-1),
    \end{equation*}
    where $\langle \textup{ }, \textup{ }\rangle$ denotes the inner product of Definition \ref{inner_product_orthogonal} and
    \begin{equation*}
        \textup{SO}(S;\mathbb{Z}):= \{g \in \textup{SL}_{n}(\mathbb{Z}) \mid g^{t}Sg = S\},
    \end{equation*}
    which is a finite group.
\end{myprop}
For our next step, namely the theta-correspondence of Eisenstein series, we need several preparations. The first one is to rewrite the Eisenstein series in a form similar to an Epstein zeta function. We are able to do this only in the case when $\Gamma_S$ has one $1$-dimensional cusp (see Definition \ref{one_dimensional_cusps}). In this case, we have Proposition \ref{main_proposition_eisenstein}.
\begin{myprop}{5.3}
    Let $S$ be such that $\#\mathcal{C}^{1}(\Gamma_S) = 1$. Then, for each $W \in \mathcal{H}_S$, there is a well-defined $R_W$ in the space of majorants of $S_1$ (see Definition \ref{majorants_defn}) such that
    \begin{equation*}
        E(W,s) = \sum_{\ell \in X/\textup{GL}_{2}(\mathbb{Z})} \left(\det(R_{W}[\ell])\right)^{-s/2},
\end{equation*}
where  $X = \left\{\m{l&m} \mid l,m \in \mathbb{Z}^{n+4}, \m{l&m} \textup{ primitive}, S_1\left[\m{l&m}\right]=0\right\}$.
\end{myprop}
Let now $Z \in \mathbb{H}_2$ (Siegel's upper half plane, see Section \ref{theta_series_section}) and $W \in \mathcal{H}_S$. We define the theta series $\Theta(Z,W)$ as in Proposition \ref{theta_definition}. This is a real analytic function that transforms as a degree $2$ Siegel modular form of weight $k=-n/2$ and a (specified) character $\chi$, under the action of the congruence subgroup $\Gamma_0^{(2)}(q)$ of $\textup{Sp}_2(\mathbb{Z})$ (see Proposition \ref{richter}). Here, $q$ denotes the level of $S$ (see Definition \ref{level}).\\\\
The next step is to consider the inner product of this theta series with a Siegel Eisenstein series for $\textup{Sp}_{2}$ and hopefully obtain $E(W,s)$. However, we cannot do this directly because the theta series contains terms that would cause the inner product integral against a hypothetical appropriate symplectic Eisenstein series to diverge. These are the summands that do not have full rank. It is therefore necessary to apply suitable differential operators in order to eliminate those terms.\\\\
It turns out that the $\textup{Sp}_2(\mathbb{R})$-invariant differential operator $R$, defined in \cite{deitmar_krieg}, can be applied in this case too. However, $\Theta$ is not invariant under $\textup{Sp}_{2}$ (or more accurately under its congruence subgroup $\Gamma_0^{(2)}(q)$) and therefore, applying $R$ to it would not be helpful. For this reason, we first apply the so-called Maass-Shimura operator $\delta_{k}^{(r)}$ (see \eqref{shimura operator}). This, when applied to a modular function, increases its weight by $2r$. As $\Theta$ transforms as a weight $-n/2$ modular form, we have to pick $r=n/4$, and therefore it is necessary to impose the condition $4 \mid n$. After taking all these factors into account, we obtain the main Theorem \ref{main theorem}.
\begin{mythm}{8.2}
    Let $\widetilde{E}(Z,\chi,s)$ denote the symplectic Siegel Eisenstein series, as in Definition \ref{siegel_eisenstein}. Let $S$ have rank $n$, with $4 \mid n$ and be such that $\#\mathcal{C}^{1}(\Gamma_S) = 1$. Let also $k=-n/2$ and $r=n/4$, as before. We then have for $\textup{Re}(s)>n+1$
    \begin{equation*}
        \left\langle \widetilde{E}(Z, \chi, (s+1)/2-r), R[\delta_{k}^{(r)}\Theta](Z,W)\right \rangle_{\Gamma_{0}^{(2)}(q)} = \xi(s)\xi(s-1)\gamma_S(s)E(W,s),
    \end{equation*}
    where $\gamma_S(s)$ is an explicit gamma factor, $\xi(s)$ is the completed zeta function and $\langle \textup{ }, \textup{ }\rangle_{\Gamma_0^{(2)}(q)}$ denotes the symplectic inner product of Definition \ref{inner_product_congruence}.
\end{mythm}
The meromorphic continuation of $\mathcal{D}_{F,G}$ to $\mathbb{C}$ follows then as a corollary (see Corollary \ref{corollary}).\\\\
As an application, when $S$ corresponds to the $E_8$ lattice, we are able to deduce a precise functional equation for the Dirichlet series (see Theorem \ref{e_8}). We choose this specific lattice because it satisfies the conditions of our main Theorem \ref{main theorem} and is the unique unimodular even lattice with only one $1$-dimensional cusp. Hence, in this case, the symplectic Eisenstein series corresponds to the full $\textup{Sp}_2(\mathbb{Z})$ and therefore a functional equation for it is available.\\\\ 
The structure of the paper is as follows: In Sections \ref{preliminaries} and \ref{fourier-jacobi forms}, we give the definitions of the main objects of consideration, i.e. modular forms, Fourier-Jacobi forms and the Dirichlet series of interest. In Section \ref{integral_representation}, we obtain the integral representation of the Dirichlet series, and in Section \ref{eisenstein_series}, we rewrite the Eisenstein series of Klingen type in the form of an Epstein zeta function. In Section \ref{theta_series_section}, we define the appropriate theta series for our case, and Section \ref{diffrential_operators} is devoted to discussing the differential operators needed in order to obtain the theta-correspondence. This is done in Section \ref{theta-correspondence}. Finally, in Section \ref{e_8-section}, we obtain a functional equation in the case of the $E_8$ lattice.
\begin{notation}
    We denote the space of $m\times n$ matrices with coefficients in a ring $R$ with $\textup{Mat}_{m,n}(R)$. If $n=m$, we often use the notation $\textup{Mat}_{n}(R)$. For any matrix $M \in \textup{Mat}_{n}(R)$, we denote by $\det(M), \textup{ tr}(M)$ the determinant and trace of $M$ respectively. By $\textup{GL}_n(R)$, we denote the matrices in $M_{n}(R)$ with non-zero determinant. By $1_n$, we denote the $n \times n$ identity matrix. For any vector $v$, we denote by $v^{t}$ its transpose. We also use the bracket notation $A[B] := \overline{B}^{t}AB$ for suitably sized complex matrices $A,B$. For a complex number $z$, we denote by $e(z) := e^{2\pi i z}$. Finally, let $\zeta(s)$ denote the Riemann zeta function, $\xi(s):=\pi^{-s/2}\Gamma(s/2)\zeta(s)$ denote the completed zeta function and $\Gamma(s):= \int_{0}^{\infty}t^{s-1}e^{-t}\textup{d}t$ denote the Gamma function.
\end{notation}
\section{Preliminaries}\label{preliminaries}
Let $V$ denote a finite-dimensional vector space over $\mathbb{Q}$. We start with the following Definition.
\begin{defn}
    A $\mathbb{Z}$-lattice $\Lambda$ is a free, finitely generated $\mathbb{Z}$-module, which spans $V$ over $\mathbb{Q}$.
\end{defn}
In the following, let $V := \mathbb{Q}^{n}$ and $L := \mathbb{Z}^{n}$ with $n \geq 1$. Assume $S$ is an even integral positive definite symmetric matrix of rank $n$. Here, even means $S[x] \in 2\mathbb{Z}$ for all $x \in L$. Let also $Q := S/2$. We define
\begin{equation*}
    S_0 := \m{&&1\\ &-S&\\1&&}, \ S_1 := \m{&&1\\&S_0&\\1&&}
\end{equation*}
of real signatures $(1,n+1)$ and $(2,n+2)$ respectively and denote by $Q_i := S_i/2$ for $i=0,1$.
Let also $V_0:= \mathbb{Q}^{n+2}$ and $V_1 := \mathbb{Q}^{n+4}$ and consider the quadratic spaces $(V_0, \phi_0), \ (V_1, \phi_1)$, where
\begin{align*}
    \phi_i : V_i \times V_i &\longmapsto \mathbb{Q}\\
    (x,y) &\longmapsto \frac{1}{2}x^{t}S_iy,
\end{align*}
for $i= 0, 1$. 
We then have that $\phi:=\phi_{0} \mid_{V\times V}$ is just $(x, y) \longmapsto -x^{t}Sy/2$, and we make the assumption that $L = \mathbb{Z}^{n}$ is a maximal $\mathbb{Z}$-lattice with respect to $\phi$. Here, maximal means that the lattice $L$ is maximal among all $\mathbb{Z}$-lattices in $V$ on which $\phi[x]$ takes values in $\mathbb{Z}$, where we use the notation $\phi[x]:=\phi(x,x)$ for $x \in V$.\\

From \cite[Lemma 6.3]{shimura2004arithmetic}, we then obtain that $L_0:= \mathbb{Z}^{n+2}$ is a $\mathbb{Z}$-maximal lattice in $V_0$.\\

If now $K \supset \mathbb{Q}$ is a field, we define the corresponding special orthogonal groups of $K$-rational points via
\begin{equation*}
    G^{*}_{K} := \{g \in \textup{SL}_{n+2}(K) \mid g^{t}S_0g = S_0\},
\end{equation*}
\begin{equation*}
    G_{K} := \{g \in \textup{SL}_{n+4}(K) \mid g^{t}S_1g = S_1\}.
\end{equation*}
We view $G^{*}_{K}$ as a subgroup of $G_{K}$ via the embedding $g \longmapsto \m{1&&\\&g&\\ &&1}$.\\
We also use the notation $\textup{SO}^{\phi_0}(V_0)$ and in general $\textup{SO}^{\phi}(V)$ for any quadratic space $(V,\phi)$.\\\\
Let now $\mathcal{H}_{S}$ denote one of the connected components of
\begin{equation*}
    \{Z \in V_0 \otimes_{\mathbb{Q}} \mathbb{C} \mid \phi_0[\textup{Im}Z] > 0\}.
\end{equation*}
In particular, if we denote by $\mathcal{P}_{S} := \{y'=(y_1,y,y_2) \in \mathbb{R}^{n+2} \mid y_1 >0, \ \phi_0[y']>0\}$, we choose:
\begin{equation*}
    \mathcal{H}_S = \{z=u+iv \in V_0 \otimes_{\mathbb{R}} \mathbb{C} \mid v \in \mathcal{P}_S\}.
\end{equation*}
For a matrix $g \in \textup{Mat}_{n+4}(\mathbb{R})$, we write it as
\begin{equation*}
    g = \m{\alpha & a^t &\beta \\ b&A&c\\ \gamma&d^t&\delta},
\end{equation*}
with $A \in \textup{Mat}_{n+2}(\mathbb{R}), \alpha,\beta,\gamma,\delta \in \mathbb{R}$ and $a,b,c,d$ real column vectors. Now the map
\begin{equation}\label{action}
    Z \longmapsto g\langle Z\rangle = \frac{-\frac{1}{2}S_0[Z]b+AZ+c}{-\frac{1}{2}S_0[Z]\gamma+d^tZ+\delta}
\end{equation}
gives a well-defined transitive action of $G_{\mathbb{R}}^{0}$ on $\mathcal{H_{S}}$, where $G_{\mathbb{R}}^{0}$ denotes the identity component of $G_{\mathbb{R}}$. The denominator of the above expression is the factor of automorphy
\begin{equation*}
    j(\gamma, Z) := -\frac{1}{2}S_0[Z]\gamma+d^tZ+\delta.
\end{equation*}
Let now $L_1 := \mathbb{Z}^{n+4}$ and define the groups
\begin{equation*}
    \Gamma(L_0):=\{g \in G_{\mathbb{Q}}^{*} \mid gL_0=L_0\}, 
\end{equation*}
\begin{equation*}
\Gamma(L_1) := \{g \in G_{\mathbb{Q}} \mid gL_1 = L_1\}.
\end{equation*}
Let also $\Gamma^{+}(L_0) := \Gamma(L_0) \cap G^{*,0}_{\mathbb{R}}$. If then
\begin{equation*}
    \Gamma_{S} := G_{\mathbb{R}}^{0} \cap \Gamma(L_1),
\end{equation*}
we have the following definition:
\begin{defn}\label{modular_form_defn}
Let $k \in \mathbb{Z}$. A holomorphic function $F : \mathcal{H}_S \longrightarrow \mathbb{C}$ is called a modular form of weight $k$ with respect to the group $\Gamma_S$ if it satisfies the equation
\begin{equation*}
    (F|_k \gamma)(Z) := j(\gamma, Z)^{-k}F(\gamma\langle Z\rangle) = F(Z)
\end{equation*}
for all $\gamma \in \Gamma_S$ and $Z \in \mathcal{H}_S$. We will denote the set of such forms by $M_{k}(\Gamma_S)$.
\end{defn}
Let us now give a couple more definitions on lattices.
\begin{defn}
    Given a $\mathbb{Z}$-lattice $\Lambda$, equipped with a bilinear form $\sigma : V \times V \longrightarrow \mathbb{Q}$, where $V :=  \Lambda \otimes_{\mathbb{Z}} \mathbb{Q}$, we define its dual lattice by
    \begin{equation*}
        \Lambda^{*} := \{x \in \Lambda \otimes_{\mathbb{Z}} \mathbb{Q} \mid \sigma(x, y) \in \mathbb{Z} \textup{ }\forall y \in \Lambda\}.
    \end{equation*}
\end{defn}
\begin{defn}\label{level}
    With the notation as above, the level of the lattice $\Lambda$ is given as the least positive integer $q$ such that $\frac{1}{2}q\sigma(x,x) \in \mathbb{Z}$ for every $x \in \Lambda^{*}$.
\end{defn}
Now, $F \in M_{k}(\Gamma_{S})$ admits a Fourier expansion of the form (cf. \cite[page 25]{eisenstein_thesis}):
\begin{equation*}
    F(Z) = \sum_{r \in L_0^{*}}A(r)e(r^{t}S_0Z),
\end{equation*}
where $Z \in \mathcal{H}_S$. It is then Koecher's principle that gives us that $A(r)=0$ unless $r \in L_0^{*} \cap \overline{\mathcal{P_S}}$ (cf. \cite[Theorem 1.5.2]{eisenstein_thesis}). Here, $\overline{\mathcal{P}_S}$ denotes the closure of $\mathcal{P_S}$. By \cite[Theorem 1.6.23]{eisenstein_thesis}, and because $L$ is $\mathbb{Z}$-maximal, we have the following definition for cusp forms.
\begin{defn}\label{fourier coefficients}
    $F \in M_{k}(\Gamma_S)$ is called a cusp form if it admits a Fourier expansion of the form
    \begin{equation*}
        F(Z) = \sum_{r \in L_0^{*} \cap \mathcal{P_S}}A(r)e(r^{t}S_0Z).
    \end{equation*}
    We denote the space of cusp forms by $S_k(\Gamma_S)$.
\end{defn}
Finally, we define a Petersson inner product, as in \cite[Remark 1.6.25]{eisenstein_thesis}.
\begin{defn}\label{inner_product_orthogonal}
    Let $\mathcal{Q}$ denote a fundamental domain for the action of $\Gamma_S$ on $\mathcal{H}_S$. Assume $F,G \in M_{k}(\Gamma_S)$ with at least one belonging in $S_{k}(\Gamma_S)$. We define their Petersson inner product as
    \begin{equation*}
        \langle F,G \rangle := \int_{\mathcal{Q}} F(Z)\overline{G(Z)}\left(Q_0[\textup{Im}Z]\right)^{k}\textup{d}^{*}Z,
    \end{equation*}
    where $\textup{d}^{*}Z = \left(Q_0[\textup{Im} Z]\right)^{-(n+2)}\textup{d}Z$ denotes the $G_{\mathbb{R}}^{0}$-invariant volume element on $\mathcal{H}_S$. We remind here again that $Q_0 = S_0/2$.
\end{defn}
\section{Fourier-Jacobi Forms}\label{fourier-jacobi forms}
In this Section, we will define Fourier-Jacobi forms of lattice index. We mainly follow Mocanu's thesis in \cite{jacobi_lattice} and Krieg in \cite{krieg_jacobi}. \\\\
For now, assume that $V$ is a vector space of dimension $n<\infty$ over $\mathbb{Q}$ together with a positive definite symmetric bilinear form $\sigma$ and an even lattice $\Lambda$ in $V$, i.e. $\sigma(\lambda, \lambda) \in 2\mathbb{Z}$ for all $\lambda \in \Lambda$. We start with the following definitions:
\begin{defn}
    We define the Heisenberg group to be:
    \begin{equation*}
        H^{(\Lambda,\sigma)}(\mathbb{R}) = \{(x,y,\zeta) \mid x,y \in \Lambda \otimes \mathbb{R}, \zeta \in S^{1}\},
    \end{equation*}
    where $S^{1}:=\{z \in \mathbb{C} \mid |z|=1\}$, equipped with the following composition law:
    \begin{equation*}
        (x_1,y_1,\zeta_1)(x_2,y_2,\zeta_2) := (x_1+x_2, y_1+y_2, \zeta_1\zeta_2e(\sigma(x_1,y_2))).
    \end{equation*}
    The integral Heisenberg group is defined to be $H^{(\Lambda,\sigma)}(\mathbb{Z}):= \{(x,y,1) \mid x,y \in \Lambda\}$ and in the following we drop the last coordinate for convenience.
\end{defn}
\begin{proposition}
    The group $\textup{SL}_2(\mathbb{R})$ acts on $H^{(\Lambda, \sigma)}(\mathbb{R})$ from the right, via
    \begin{equation*}
        ((x,y,\zeta), A) \longmapsto (x,y,\zeta)^{A} := \left((x,y)A, \zeta e\left(\sigma[(x,y)A]-\frac{1}{2}\sigma(x,y)\right)\right).
    \end{equation*}
    where $(x,y)A$ denotes the formal multiplication of the vector $(x,y)$ with $A$, i.e. if $A =\m{a&b\\c&d}$, we have $(x,y)A := (ax+cy, bx+dy)$.
\end{proposition}
\begin{defn}
    The real Jacobi group associated with $(\Lambda, \sigma)$, denoted by $J^{(\Lambda, \sigma)}(\mathbb{R})$, is defined to be the semi-direct product of $\textup{SL}_2(\mathbb{R})$ and $H^{(\Lambda, \sigma)}(\mathbb{R})$. The composition law is then
    \begin{equation*}
        (A,h)\cdot (A',h') := (AA', h^{A'}h').
    \end{equation*}
    We also define the integral Jacobi group to be the semi-direct product of $\textup{SL}_{2}(\mathbb{Z})$ and $H^{(\Lambda, \sigma)}(\mathbb{Z})$ and we will denote it by $J^{(\Lambda, \sigma)}$.
\end{defn}
We are now going to define some slash operators, acting on holomorphic, complex-valued functions on $\mathbb{H} \times (\Lambda \otimes \mathbb{C})$.
\begin{defn}\label{k-actions}
    Let $k$ be a positive integer and $f : \mathbb{H} \times (\Lambda \otimes \mathbb{C}) \longrightarrow \mathbb{C}$ a holomorphic function. For $M =\m{a&b\\c&d} \in \textup{SL}_{2}(\mathbb{R})$ , we define:
    \begin{equation*}
        \left(f|_{k, (\Lambda, \sigma)} [M]\right)(\tau, z) := (c\tau+d)^{-k}e^{-\pi i c\sigma(z,z)/(c\tau + d)}f\left(\frac{a\tau+b}{c\tau+d}, \frac{z}{c\tau+d}\right).
    \end{equation*}
    In the case when $M \in \textup{GL}_2^{+}(\mathbb{R})$, we use $\det(M)^{-1/2}M$ instead of $M$. For $h = (x,y,\zeta) \in H^{(\Lambda, \sigma)}(\mathbb{R})$:
    \begin{equation*}
        \left(f \mid_{k, (\Lambda,\sigma)}h\right)(\tau,z): = \zeta \cdot e^{\pi i \tau \sigma(x, x)+2\pi i \sigma(x, z)}f(\tau, z+x\tau+y).
\end{equation*}
    Finally, for the action of $J^{(\Lambda, \sigma)}(\mathbb{R})$ on complex-valued, holomorphic functions on $\mathbb{H} \times (L \otimes \mathbb{C})$, we have:
    \begin{equation*}
        (f, (A,h)) \longmapsto \left(f \mid_{k, (\Lambda, \sigma)}(A,h)\right)(\tau, z) := \left((f \mid_{k, (\Lambda, \sigma)} A) \mid_{k, (\Lambda, \sigma)} h\right)(\tau,z).
    \end{equation*}
\end{defn}
We now have the following Definition (\cite[Definition 1.23]{jacobi_lattice}):
\begin{defn}\label{fourier-jacobi-defn}
Let $V_{\mathbb{C}} := V \otimes \mathbb{C}$ and extend $\sigma$ to $V_{\mathbb{C}}$ by $\mathbb{C}$-linearity. For $k$ a positive integer, a holomorphic function $f: \mathbb{H} \times V_{\mathbb{C}} \rightarrow \mathbb{C}$ (where $\mathbb{H}$ denotes the usual upper half plane) is called a Jacobi form of weight $k$ with respect to $(\Lambda, \sigma)$ if the following hold:
\begin{itemize}
    \item For all $\gamma \in J^{(\Lambda, \sigma)}$, and $(\tau,z) \in \mathbb{H} \times V_{\mathbb{C}}$, we have
    \begin{equation*}
        \left(f \mid_{k, (\Lambda, \sigma)} \gamma \right) (\tau,z) = f(\tau,z).
    \end{equation*}
    \item $f$ has a Fourier expansion of the form
\begin{equation*}
    f(\tau, z) = \sum_{m \in \mathbb{Z}, r \in \Lambda^{*}, 2m\geq \sigma[r]} c_{f}(m,r)e(m\tau + \sigma(r, z)).
\end{equation*}
\end{itemize}
We denote the space of such forms by $J_{k}(\Lambda, \sigma)$. We say $f$ is a Jacobi cusp form if $c_f(m,r)=0$, when $2m= \sigma[r]$. We denote the space of Fourier-Jacobi cusp forms by $S_{k}(\Lambda, \sigma)$.
\end{defn}
We now have a notion of a scalar product for elements of $S_{k}(\Lambda, \sigma)$ (\cite[Definition 1.33]{jacobi_lattice}). 
\begin{defn}\label{inner product_fourier-jacobi}
    Let $\phi, \psi \in S_{k}(\Lambda, \sigma)$. If $U \leq J^{(\Lambda,\sigma)}$ of finite index, we define their Petersson inner product via:
    \begin{equation*}
        \langle \phi, \psi\rangle_{U} := \frac{1}{\left[J^{\Lambda}:U\right]}\int_{U \backslash \mathcal{H}\times (\Lambda \otimes \mathbb{C})} \phi(\tau,z)\overline{\psi(\tau, z)}v^{k}e^{-2\pi \sigma(y,y)v^{-1}}\textup{d}V,
    \end{equation*}
    where $\tau = u+iv$, $z = x+iy$ and $\textup{d}V := v^{-n-2}\textup{d}u\textup{d}v\textup{d}x\textup{d}y$. This inner product does not depend on the choice of $U$, so in what follows, we drop the subscript.
\end{defn}
We now specify to our case by taking $\Lambda = L = \mathbb{Z}^{n}$ and $\sigma(x,y) = x^{t}Sy$ for all $x,y \in V$. We also write $J_{S}$ for the integral Jacobi group in this case, i.e. $J_S := \textup{SL}_{2}(\mathbb{Z}) \rtimes (\mathbb{Z}^{n} \times \mathbb{Z}^{n})$.\\

Let us discuss the Fourier-Jacobi expansion of orthogonal cusp forms of weight $k$ with respect to $\Gamma_{S}$. If we write $Z = (\omega, z, \tau) \in \mathcal{H}_{S}$ with $\omega,\tau \in \mathbb{C}, z \in \mathbb{C}^{n}$, we have that for any $m \in \mathbb{Z}$ (\cite[page 244]{gritsenko_jacobi}):
\begin{equation*}
    F(\omega+m, z, \tau) = F(\omega, z, \tau)
\end{equation*}
(by the form of elements belonging to $\Gamma_S$). Hence, we can write
\begin{equation*}
    F(Z) = \sum_{m \geq 1}\phi_m(\tau, z)e^{2\pi i m \omega},
\end{equation*}
and we call the functions $\phi_m(\tau, z)$ the Fourier-Jacobi coefficients of $F$. We note that then $\phi_{m} \in J_{k}(\mathbb{Z}^{n}, m\sigma)$. We now want to view Fourier-Jacobi forms as automorphic forms under the action of a parabolic subgroup.
\begin{defn}\label{parabolic}
    The parabolic subgroup of $\Gamma_{S}$ fixing the two-dimensional isotropic subspace spanned by $e_1,e_2$ (standard basis vectors) is defined by
    \begin{equation*}
        \Gamma_{S,J} := \left\{\m{*&*\\0&D} \in \Gamma_{S} \mid D \in \textup{Mat}_{2}(\mathbb{Z})\right\}.
    \end{equation*}
\end{defn}
Now, there is an embedding $\iota$
\begin{align*}
    \iota : J_{S} &\longrightarrow \Gamma_{S,J}\\
    (D, [x,y]) &\longmapsto M_{D} \cdot H_{x,y},
\end{align*}
where
\begin{equation*}
    H_{x,y}:= \m{1&0&y^{t}S&0&\frac{1}{2}S[y]\\0&1&x^{t}S&\frac{1}{2}S[x]&x^{t}Sy\\0&0&1_{n}&x&y\\0&0&0&1&0\\0&0&0&0&1} \textup{, }  M_{D}:= \m{D^{*}&0&0\\0&1_{n}&0\\0&0&D},  
\end{equation*}
with $D^{*}:= D\left[\m{-1&0\\0&1}\right]$ (cf. \cite[page 44]{eisenstein_thesis}). We denote by $\Gamma_{S,J}^{\bullet} := \iota(J_S)$. By \cite[Proposition 2.2.7]{ajouz}, we have that the action of $J_{S}$ on $\mathbb{H} \times \mathbb{C}^{n}$ is given by
\begin{equation*}
    ((D, [x,y]), (\tau, z)) \longmapsto \left(D\tau, \frac{z+x\tau+y}{j'(D,\tau)}\right),
\end{equation*}
where $j'$ denotes the usual factor of automorphy for the action of $\textup{SL}_{2}(\mathbb{Z})$ on $\mathbb{H}$.\\\\
Now, if we take an element $M = \iota((D, [x,y])) \in \Gamma_{S,J}^{\bullet}$, we can see that its action on an element $(\omega, z, \tau) \in \mathcal{H}_S$
\begin{align*}
    M \langle (\omega, z, \tau) \rangle = \left(*, \frac{z+x\tau+y}{j(D,\tau)}, D\tau \right)
\end{align*}
is the same as the action of $J_{S}$ on $\mathbb{H} \times \mathbb{C}^{n}$. We now have the following Proposition regarding the fundamental domain of the action of $\Gamma_{S,J}^{\bullet}$ on $\mathcal{H}_S$, which will be useful later.
\begin{proposition}\label{fundamental domain}
For $Z = (\omega,z,\tau) \in \mathcal{H}_S$, we write $\omega = x_1+iy_1, z = u+iv, \tau = x_2+iy_2$. A valid choice for the fundamental domain of the action of $\Gamma_{S,J}^{\bullet}$ on $\mathcal{H}_S$ is 
\begin{equation*}
\mathcal{F}^{J} := \left\{ Z = (\omega,z,\tau) \in \mathcal{H}_{S} \mid (z,\tau) \in \mathcal{F}, y_1y_2 - \frac{1}{2}S[v]>0, \  -\frac{1}{2} \leq x_1 \leq \frac{1}{2} \right\},
\end{equation*}
where $\mathcal{F}$ is a fundamental domain for the action of $J_{S}$ on $\mathbb{H} \times \mathbb{C}^n$.
\end{proposition}
\begin{proof}
Let $Z = (\omega, z, \tau) \in \mathcal{H}_S$. We can then pick $g \in \Gamma_{S,J}^{\bullet}$ such that $\iota^{-1}(g) \in J_{S}$ and that $Z' := g \langle Z\rangle = (\omega',z,\tau)$ with $(z,\tau) \in \mathcal{F}$ and $\omega' \in \mathbb{H}$ arbitrary. This follows from the fact that the actions are the same, as we have shown above. If now
\begin{equation*}
    T_{\lambda} = \m{1&-\lambda^{t}S_0& -Q_0[\lambda] \\ 0 & 1_{n+2}&\lambda \\ 0&0&1},
\end{equation*}
for $\lambda \in \mathbb{Z}^{n+2}$, then $T_{\lambda} \in \Gamma_{S}$ (cf. \cite[Theorem 1.4.4]{eisenstein_thesis}).
We then have $T_{\lambda}\langle Z'\rangle = Z'+\lambda$ and so we can act with a suitable $T_{\lambda}$, so that $-1/2 \leq x_1 \leq 1/2$. Finally, the condition $y_1y_2 - S[v]/2>0$ follows from the fact that $Z \in \mathcal{H}_S$ and the definition of $\mathcal{P}_{S}$ in Section \ref{preliminaries}. 
\end{proof}
\section{Dirichlet Series and Integral Representation}\label{integral_representation}
In this Section, we define the Dirichlet series of interest and give its integral representation in terms of an Eisenstein series of Klingen type.\\\\
Let $F,G \in S_{k}(\Gamma_S)$ with corresponding Fourier-Jacobi coefficients $\{\phi_m\}_{m=1}^{\infty}, \{\psi_{m}\}_{m=1}^{\infty}$. We then define the following Dirichlet series:
\begin{equation*}
    \mathcal{D}_{F,G}(s) := \sum_{m=1}^{\infty}\langle \phi_m, \psi_m \rangle m^{-s}.
\end{equation*}
Here, $\langle \textup{ }, \textup{ }\rangle$ denotes the inner product of Definition \ref{inner product_fourier-jacobi}.
\begin{lemma}\label{convergence_dirichlet}
    $\mathcal{D}_{F,G}(s)$ converges absolutely for $\textup{Re}(s) > k+1$ and represents a holomorphic function on this domain.
\end{lemma}
\begin{proof}
    The proof is similar to \cite[Lemma 1]{kohnen_skoruppa}. We will show for $N \geq 1$ that
    \begin{equation*}
        \left \langle \phi_N, \psi_N \right \rangle = \mathcal{O}(N^{k}),
    \end{equation*}
    with the constant depending only on $F, G$. Indeed, fix $(\tau, z) \in \mathbb{H} \times \mathbb{C}^{n}$ and write $\tau = u+iv, \textup{ } z = x+iy$. If $F(q) = \sum_{N=1}^{\infty}\phi_{N}(\tau,z)q^{N}$, with $q=e^{2\pi i \omega}$, we have by Cauchy's integral formula that
    \begin{equation*}
        \phi_{N}(\tau,z) = \oint_{|q|=r}\frac{F(q)}{q^{N+1}}\textup{d}q,
    \end{equation*}
    for any $0 < r < e^{-2\pi Q[y]/v}$. The bounds follow from the fact that $Q_0[\textup{Im}Z]>0$ (here $Q=S/2$ and $Z = (\omega,z,\tau) \in \mathcal{H}_S$). If now $\omega = u'+iv'$, the integral can be written as
    \begin{equation*}
        \phi_{N}(\tau,z) = \int_{0}^{1}F(Z)e^{-2\pi i N \omega}\textup{d}u',
    \end{equation*}
    for any $v' > Q[y]/v$. But now $\left|F(Z)\right|\left(Q_0[\textup{Im}Z]\right)^{k/2}$ is bounded on $\mathcal{H}_S$ from \cite[II, Lemma 3.28]{hauffe21}, say by a constant $C>0$. Therefore, after choosing $v'=Q[y]/v+1/N$, we have
    \begin{equation*}
        \left|\phi_{N}(\tau,z)\right| \leq Ce^{2\pi}\int_{0}^{1}\left(Q_0[\textup{Im}Z]\right)^{-k/2}e^{2\pi N Q[y]/v}\textup{d}u' = Ce^{2\pi}\left(\frac{v}{N}\right)^{-k/2}e^{2 \pi N Q[y]/v}.
    \end{equation*}
    Similarly for $\psi_N$ and then the claim follows from the definition of the inner product in Definition \ref{inner product_fourier-jacobi}.
\end{proof}
In order now to give an integral representation for $\mathcal{D}_{F,G}$, we need to define an appropriate Eisenstein series of Klingen type, in an analogous way to \cite[Section 3]{eisenstein_thesis}.
\begin{defn}\label{eisentein_defn}
    Let $Z \in \mathcal{H}_S$ and $s \in \mathbb{C}$ with $\textup{Re}(s) > n+1$. We define the real analytic Eisenstein series of Klingen type to be
    \begin{equation*}
        E(Z,s) := \sum_{\gamma \in \Gamma_{S,J} \backslash \Gamma_{S}} \left(\frac{Q_0[\textup{Im}(\gamma Z)]}{\textup{Im}((\gamma Z)_2)}\right)^s,
    \end{equation*}
    where for $Z = (\omega, z, \tau) \in \mathcal{H}_S$, we write $Z_2 := \tau$. Also, recall $Q_0 = S_0/2$.
\end{defn}
\begin{proposition}
    $E(Z,s)$ is well defined, is invariant under the action of $\Gamma_S$ and converges absolutely and uniformly whenever $Z$ belongs to a compact subset of $\mathcal{H}_S$ and $s$ satisfies $\textup{Re}(s) > n+1$.
\end{proposition}
\begin{proof}
    Let $\gamma \in \Gamma_{S,J}$. We write
    \begin{equation*}
        \gamma = \m{*&*\\0&D}, \textup{ } D \in \textup{SL}_{2}(\mathbb{Z}),
    \end{equation*}
    (cf. \cite[Proposition 3]{krieg_integral_orthogonal}). Now, for $Z = (\omega, z, \tau) \in \mathcal{H}_S$, we have $\gamma \langle Z\rangle_2 = D \langle \tau \rangle$, where the action on the right denotes the usual action of $\textup{SL}_{2}(\mathbb{Z})$ on $\mathbb{H}$ (cf. \cite[page 115]{eisenstein_thesis}). By \cite[Lemma 3.20]{Borcherds_Products}, we have
    \begin{equation*}
        Q_0[\textup{Im}(\gamma Z)] = |j(\gamma,Z)|^{-2} Q_0[\textup{Im}Z],
    \end{equation*}
    and so
    \begin{equation*}
        \frac{Q_0[\textup{Im}(\gamma Z)]}{\textup{Im}((\gamma Z)_2)}= \frac{1}{|j(\gamma,Z)|^2} \frac{Q_0[\textup{Im}Z]}{\textup{Im}(D \tau)} = \frac{|j'(D,\tau)|^2}{|j(\gamma,Z)|^2} \frac{Q_0[\textup{Im}Z]}{\textup{Im}(Z_2)},
    \end{equation*}
    where again $j'$ denotes the usual factor of automorphy for the action of $\textup{SL}_{2}(\mathbb{Z})$ on $\mathbb{H}$. But we have $j(\gamma, Z) = j'(D,\tau)$ and hence $E(Z,s)$ is well-defined.\\\\
    The invariance under $\Gamma_{S}$ is clear. For the convergence, we can write
    \begin{equation*}
        E(Z,s) = \sum_{\gamma \in \Gamma_{S,J}\backslash \Gamma_{S}} \left(Q_0[\textup{Im}Z]\right)^{s}\left(\textup{Im} (\gamma Z)_{2}\right)^{-s}\left|j(\gamma, Z)\right|^{-2s}.
    \end{equation*}
    But by the proof of \cite[Theorem 3.1.1]{eisenstein_thesis}, the sum
    \begin{equation*}
        \sum_{\gamma \in \Gamma_{S,J} \backslash \Gamma_{S}}\left(\textup{Im} (\gamma Z)_{2}\right)^{-k/2}\left|j(\gamma, Z)\right|^{-k}
    \end{equation*}
    converges locally uniformly whenever $k>2n+2$. From this, the claim follows.
\end{proof}
We are now ready to give the main Proposition of this Section.
\begin{proposition}\label{integral_representation_dirichlet}
    Let $F,G \in S_k(\Gamma_S)$. For $Z \in \mathcal{H}_S$ and $s \in \mathbb{C}$ with $\textup{Re}(s) > n+2$, we have
    \begin{equation*}
        \langle F(Z)E(Z,s), G(Z)\rangle = \frac{1}{\#\textup{SO}(S;\mathbb{Z})}(4\pi)^{-(s+k-n-1)} \Gamma(s+k-n-1) \mathcal{D}_{F,G}(s+k-n-1),
    \end{equation*}
    where $\langle \textup{ }, \textup{ }\rangle$ denotes the inner product of Definition \ref{inner_product_orthogonal}, and
    \begin{equation*}
        \textup{SO}(S;\mathbb{Z}):= \{g \in \textup{SL}_{n}(\mathbb{Z}) \mid g^{t}Sg = S\},
    \end{equation*}
    which is finite.
\end{proposition}
\begin{proof}
Let $I(s) := \langle F(Z)E(Z,s)\textup{, } G(Z)\rangle$. If we denote by $\mathcal{Q}$ a fundamental domain for the action of $\Gamma_{S}$ on $\mathcal{H_S}$, by using the standard unfolding argument, we have for $\textup{Re}(s)>n+1$
\begin{align*}
I(s) &= \int_{\mathcal{Q}} F(Z) \overline{G(Z)} \sum_{\gamma \in \Gamma_{S,J} \setminus \Gamma_S} \left(\frac{Q_0[\textup{Im}(\gamma Z)]}{\textup{Im}((\gamma Z)_2)}\right)^s \left(Q_0[\textup{Im}Z] \right)^k \textup{d}^{*}Z\\
&=\frac{1}{[\Gamma_{S,J} : \Gamma_{S,J}^{\bullet}]}\int_{\mathcal{F}^{J}} F(Z) \overline{G(Z)}\left(\frac{Q_0[Y]}{\textup{Im}(\tau)}\right)^s \left(Q_0[Y] \right)^{k-n-2} \textup{d}X\textup{d}Y,
\end{align*}
where $\mathcal{F}^J$ is a fundamental domain for the action of $\Gamma_{S,J}^{\bullet}=\iota(J_S)$ on $\mathcal{H}_{S}$ and $Z=(\omega, z, \tau)= X+iY$. We note here 
\begin{equation*}
    [\Gamma_{S,J} : \Gamma_{S,J}^{\bullet}] = \#\textup{SO}(S;\mathbb{Z}) < \infty,
\end{equation*}
because $S$ is positive definite (cf. \cite[Section 1.7]{eisenstein_thesis}). Hence, from Proposition \ref{fundamental domain}, with the same notation as there, we get

\begin{equation*}
I(s) = \frac{1}{\#\textup{SO}(S;\mathbb{Z})}\int_{\mathcal{F}} \int_{y_1y_2 - \frac{1}{2}S[v]>0} \int_{-1/2 \leq x_1 \leq 1/2} F(Z) \overline{G(Z)}  
 y_2^{-s} \left( y_1y_2 - \frac{1}{2}S[v] \right)^{k-n-2+s} \textup{d}X\textup{d}Y,
\end{equation*}
where $\mathcal{F}$ is a fundamental domain of the action of $J_{S}$ on $\mathbb{H} \times \mathbb{C}^n$. We now write 
\begin{equation*}
    \displaystyle{F(Z)= \sum_{m=1}^{\infty} \phi_m(\tau,z) e^{2 \pi i m \omega}}, \,\,\, \displaystyle{G(Z)=\sum_{m=1}^{\infty}\psi_m(\tau,z)e^{2\pi i m \omega}}.
\end{equation*}
Using the fact that for integers $m_1, m_2$, we have
\begin{equation*}
\int_{-1/2}^{1/2} e^{2\pi i (m_1-m_2)x_1} \textup{d}x_1 =
\begin{cases}
    1 & \textup{ if } m_1=m_2\\
    0 & \textup{ otherwise}
\end{cases},
\end{equation*}
we get
\begin{multline*}
I(s)= \frac{1}{\#\textup{SO}(S;\mathbb{Z})}\int_{\mathcal{F}} \int_{y_1y_2 - \frac{1}{2}S[v]>0} \sum_{m=1}^{\infty} \phi_m(\tau,z)\overline{\psi_m(\tau,z)} e^{-4\pi m y_1} y_2^{-s} \times \\\times \left( y_1y_2 - \frac{1}{2}S[v] \right)^{k-n-2+s} \textup{d}X\textup{d}Y.
\end{multline*}
We now set $t:=y_1 - \frac{S[v]}{2y_2}$, or equivalently $y_2t = y_1y_2 -\frac{1}{2}S[v]$. We then get
\begin{equation*}
I(s)=\frac{1}{\#\textup{SO}(S;\mathbb{Z})}\int_{\mathcal{F}} \int_{t=0}^{\infty} \sum_{m=1}^{\infty} \phi_m(\tau,z)\overline{\psi_m(\tau,z)} e^{-4\pi m t} e^{-2\pi m \frac{S[v]}{y_2}} y_2^{-s} \left( ty_2 \right)^{k-n-2+s} \textup{d}t \textup{d}u \textup{d}v \textup{d}x_2 \textup{d}y_2.
\end{equation*}
But
\begin{equation*}
\int_{t=0}^{\infty} e^{-4\pi m t} t^{k-n-2+s} \textup{d}t= \Gamma(s+k-n-1) (4\pi)^{-(s+k-n-1)} m^{-(s+k-n-1)}.
\end{equation*}
Moreover, in this case, the inner product of Definition \ref{inner product_fourier-jacobi} reads as:
\begin{equation*}
\langle\phi_m,\psi_m\rangle = \int_{\mathcal{F}} \phi_m(\tau,z) \overline{\psi_m(\tau,z)} y_2^{k-n-2} e^{-\frac{2 \pi m}{y_2} S[v]} \textup{d}u \textup{d}v \textup{d}x_2 \textup{d}y_2.
\end{equation*}
Putting all the above together, we obtain
\begin{equation*}
I(s) = \frac{1}{\#\textup{SO}(S;\mathbb{Z})}(4\pi)^{-(s+k-n-1)} \Gamma(s+k-n-1) \mathcal{D}_{F,G}(s+k-n-1),
\end{equation*}
or equivalently
\begin{equation*}
\langle F(Z)E(Z,s), G(Z)\rangle = \frac{1}{\#\textup{SO}(S;\mathbb{Z})}(4\pi)^{-(s+k-n-1)} \Gamma(s+k-n-1) \mathcal{D}_{F,G}(s+k-n-1),
\end{equation*}
as claimed.
\end{proof}
\section{Manipulation of the Eisenstein Series}\label{eisenstein_series}
It is now clear from the above that the analytic properties of the Dirichlet series of interest reduce to the ones of the Klingen Eisenstein series, as given in Definition \ref{eisentein_defn}. This Section is devoted to writing this Eisenstein series in the form of an Epstein zeta function, similar to \cite[equation (7)]{krieg1991}. In our case, because of the form of the Eisenstein series, we cannot use the method of Krieg in \cite{krieg1991} with the minors of the determinant. It turns out we can write the Eisenstein series in such a form, provided that the number of one-dimensional cusps is $1$. 
\begin{defn}[Definition 1.6.17, \cite{eisenstein_thesis}]\label{one_dimensional_cusps}
    The set of $\Gamma_S$-orbits of one-dimensional cusps is defined by
    \begin{equation*}
        \mathcal{C}^{1}(\Gamma_S) := \left\{\Gamma_S W \mid W \textup{ is a two-dimensional isotropic plane in } V_1\right\}.
    \end{equation*}
    A two-dimensional isotropic plane, or isotropic plane, is defined by two linearly independent vectors $g,h \in V_1$. We normalise $g,h$ such that $g,h \in L_1^{*}$ and such that $\gcd{(S_1g)}=\gcd{(S_1h)}=1$. The isotropy condition means $S_1\left[\m{g&h}\right]=0$.
\end{defn}
We also have the following Definition of the majorants for $S_1$.
\begin{defn}\label{majorants_defn}
    The space of majorants for $S_1$ is defined by
    \begin{equation*}
        \mathfrak{H} := \{R \in M_{n+4}(\mathbb{R}) \mid R = R^t>0, \ RS_1^{-1} R = S_1\}.
    \end{equation*}
\end{defn}
The main Proposition of the Section is the following:
\begin{proposition}\label{main_proposition_eisenstein}
    Let $S$ be such that $\#\mathcal{C}^{1}(\Gamma_S) = 1$. Then, for each $Z \in \mathcal{H}_S$, there is $R_Z \in \mathfrak{H}$ such that
    \begin{equation*}
        E(Z,s) = \sum_{\gamma \in \Gamma_{S,J}\backslash\Gamma_S} \left(\frac{\textup{Im} (\gamma Z)_2}{Q_0[\textup{Im}(\gamma Z)]}\right)^{-s} = \sum_{\ell \in X/\textup{GL}_{2}(\mathbb{Z})} \left(\det(R_{Z}[\ell])\right)^{-s/2},
\end{equation*}
where  
\begin{equation*}
    X = \left\{\m{l&m} \mid l,m \in \mathbb{Z}^{n+4}, \m{l&m} \textup{ primitive},\  S_1\left[\m{l&m}\right]=0\right\}.
\end{equation*}
\\Here, a matrix being primitive means that its elementary divisors are all $1$ (see \cite[Section 3]{Shimura_Euler_Product}).
\end{proposition}
The rest of the Section is devoted to proving this Proposition.  We start with two Lemmas regarding the elements of $G_{\mathbb{R}}$, i.e., the special orthogonal group attached to $S_1$.
\begin{lemma}\label{inverse}
   Let $\gamma \in G_{\mathbb{R}}$ and write $\gamma = \m{*&l&m}^{t}$ with $l,m \in \mathbb{R}^{n+4}$.
Then
\begin{equation*}
    \gamma^{-1} = \m{S_1^{-1}m & S_1^{-1}l&*}.
\end{equation*}
\end{lemma}
\begin{proof}
    The proof follows from the fact that if
    \begin{equation*}
        \gamma = \m{\alpha & a^t &\beta \\ b&A&c\\ \gamma&d^t&\delta} \in G_{\mathbb{R}},
    \end{equation*}
    then from the relation $S_1[\gamma] = S_1$, we have
    \begin{equation*}
        \gamma^{-1} = \m{\delta & c^tS_0 &\beta \\ S_0^{-1}d&S_0^{-1}A^{t}S_0&S_0^{-1}a\\ \gamma&b^tS_0&\alpha}.
    \end{equation*}
\end{proof}
\begin{lemma}\label{lem: conditions}
Let $\gamma \in G_{\mathbb{R}}$ and write $\gamma = \m{*\\l^{t}\\m^{t}} = \m{*&*&*&*&*\\l_1&l_2&x^{t}&l_{n+3}&l_{n+4}\\m_1&m_2&y^t&m_{n+3}&m_{n+4}}$. Then, if
\begin{equation*}
    \ell := S_1^{-1}\m{m&l} = \m{m_{n+4}&l_{n+4}\\m_{n+3}&l_{n+3}\\-S^{-1}y&-S^{-1}x\\m_2&l_2\\m_1&l_1},
\end{equation*}
we have $S_1[\ell] = 0$.
\end{lemma}
\begin{proof}
Let $M = \m{e_1&e_2}$, where $e_1,e_2$ are the two standard basis vectors of $\mathbb{R}^{n+4}$. Then
\begin{equation*}
    S_1[\gamma^{-1}M] = S_1[\gamma^{-1}][M] = S_1[M] = 0,
\end{equation*}
because $\gamma^{-1} \in G_\mathbb{R}$ and the subspace generated by $M$ is totally isotropic. The result then follows by Lemma \ref{inverse} because $\ell$ is just the matrix formed by the first two columns of $\gamma^{-1}$, i.e. $\gamma^{-1}M$.
\end{proof}
Now, for any $Z \in \mathcal{H}_S$, we want to choose a specific majorant for $S_1$, so that the terms of the Eisenstein series take the form in Proposition \ref{main_proposition_eisenstein}. We start with defining the majorant for the element $I := (i,0\cdots, 0, i)^{t} \in \mathcal{H}_S$.
\begin{lemma}\label{identity}
    Let $I := (i, 0,\cdots,0, i)^{t}$ and $R_{I} := \textup{diag}(1,1,S,1,1)$. Then, if $\gamma \in G_{\mathbb{R}}$, with $\gamma = \m{*&l&m}^{t}$, we have
\begin{equation*}
    \left(\frac{\textup{Im} ((\gamma I)_2)}{Q_0[\textup{Im}(\gamma I)]}\right)^2 = \det{(R_I[\ell])},
\end{equation*}
where $\ell = S_1^{-1}\m{m&l}$, as in Lemma \ref{lem: conditions}.
\end{lemma}
\begin{proof}
We have $S_0[I] = -2$ and therefore
\begin{multline*}
    \textup{Im}((\gamma I)_2) = \textup{Im}\left(\frac{l_1+(l_2+l_{n+3})i+l_{n+4}}{m_1+(m_2+m_{n+3})i+m_{n+4}}\right)= \\ =  \frac{(l_2+l_{n+3})(m_1+m_{n+4})-(l_1+l_{n+4})(m_2+m_{n+3})}{(m_1+m_{n+4})^2+(m_2+m_{n+3})^2}.
\end{multline*}
Also,
\begin{equation*}
    Q_0[\textup{Im}(\gamma I)] = |j(\gamma, I)|^{-2}Q_0[\textup{Im}I] = \frac{1}{(m_1+m_{n+4})^2+(m_2+m_{n+3})^2}.
\end{equation*}
Hence,
\begin{equation*}
    \frac{\textup{Im} ((\gamma I)_2)}{Q_0[\textup{Im}(\gamma I)]} = (l_2+l_{n+3})(m_1+m_{n+4})-(l_1+l_{n+4})(m_2+m_{n+3}).
\end{equation*}
Now, from Lemma \ref{lem: conditions}, with $\ell$ written as there, we have 
\begin{equation*}
    \begin{cases}
    S^{-1}[x] = 2(l_{n+4}l_1+l_{n+3}l_2). \\
    S^{-1}[y] = 2(m_{n+4}m_1+m_{n+3}m_2). \\ 
    x^tS^{-1}y = l_{n+4}m_1+l_{n+3}m_2+l_{2}m_{n+3}+l_1m_{n+4}.
\end{cases}
\end{equation*}
By then computing $R_I[\ell]$ and taking the determinant, the result follows.\end{proof}

After defining the majorant for $I$, we can use the transitivity of the action defined in \eqref{action}, in order to define a majorant for every $Z$ in $\mathcal{H}_S$.
\begin{proposition}\label{prop: majorants}
Let $Z \in \mathcal{H}_S$. Then, $\exists R_{Z} \in \mathfrak{H}$ such that for all $\gamma \in G_{\mathbb{R}}$,
    \begin{equation*}
        \left(\frac{\textup{Im} ((\gamma Z)_2)}{Q_0[\textup{Im}(\gamma Z)]}\right)^2 = \det(R_{Z}[\ell]),
    \end{equation*}
where $\ell$ is the matrix formed by the first two columns of $\gamma^{-1}$.
\end{proposition}
\begin{proof}
    We start the proof by constructing such an $R_{Z}$. We denote by $I = (i, 0,\cdots,0, i)^{t}$. By transitivity, $\exists \delta \in G_{\mathbb{R}}^{0}$ such that $\delta \langle I\rangle  = Z$. We then define $R_{Z} := R_{I}[\delta^{-1}]$ and we claim this is well-defined. We prove this in Lemma \ref{well-defined}. Then
\begin{equation*}
    \left(\frac{\textup{Im} ((\gamma Z)_2)}{Q_0[\textup{Im}(\gamma Z)]}\right)^2 = \left(\frac{\textup{Im} ((\gamma\delta I)_2)}{Q_0[\textup{Im}(\gamma\delta I)]}\right)^2 = \det \left(R_{I}[\ell]\right),
\end{equation*}
where if $\gamma \delta = \m{*\\l^{t}\\m^{t}}$, we have $\ell = S_1^{-1}\m{m & l}$, by Lemma \ref{identity}. 
Now,
\begin{equation*}
    R_{Z} = R_{I}[\delta^{-1}] \implies R_{I} = R_{Z}[\delta].
\end{equation*}
Hence
\begin{equation*}
     \left(\frac{\textup{Im} ((\gamma Z)_2)}{Q_0[\textup{Im}(\gamma Z)]}\right)^2 = \det(R_{Z}[\delta\ell]).
\end{equation*}
Now, if $\gamma = \m{*\\(l')^{t}\\(m')^{t}}$, we want to show that the above quotient equals $\det(R_{Z}[S_1^{-1}\m{m'&l'}])$. But
\begin{equation*}
    \gamma \delta = \m{*\\l^t\\m^t} \implies \gamma = \m{* \\ l^t \\ m^t}\delta^{-1} = \m{* \\ l^t\delta^{-1}\\m^t\delta^{-1}} = \m{*\\ ((\delta^{-1})^tl)^{t}\\((\delta^{-1})^tm)^{t}},
\end{equation*}
so 
\begin{equation*}
    \m{m'&l'} = \m{(\delta^{-1})^tm&(\delta^{-1})^tl}.
\end{equation*}
Suffices to then show that (we remind here that $\ell = S_1^{-1}\m{m&l}$)
\begin{equation*}
    S_1^{-1}(\delta^{-1})^{t} = \delta S_1^{-1} \iff \delta S_1^{-1} \delta^{t} = S_1^{-1} \iff (\delta^{-1})^tS_1\delta^{-1} = S_1 \iff S_1[\delta^{-1}] = S_1,
\end{equation*}
which is true. \\\\
The only thing remaining to show is that $R_{Z} \in \mathfrak{H}$. But as $R_I$ is symmetric and positive definite, the same holds for $R_Z = R_I[\delta^{-1}]$. Finally, it is easy to show that $R_{I}S_1^{-1}R_I = S_1$ and then
\begin{equation*}
    R_{Z}S_1^{-1}R_Z = (\delta^{-1})^{t}R_I\delta^{-1}S_1^{-1}(\delta^{-1})^{t}R_I\delta^{-1} = (\delta^{-1})^{t}R_IS_1^{-1}R_I\delta^{-1}= (\delta^{-1})^{t}S_1\delta^{-1} = S_1[\delta^{-1}] = S_1.
\end{equation*}
Hence, $R_Z \in \mathfrak{H}$, as required.
\end{proof}
We now give the final Lemma, showing that $R_{Z}$ is well-defined.
\begin{lemma}\label{well-defined}
    For each $Z \in \mathcal{H}_S$, $R_{Z}$, as defined in the proof of Proposition \ref{prop: majorants}, is well-defined.
\end{lemma}
\begin{proof}
    We want to show that for $\delta_1,\delta_2 \in G_{\mathbb{R}}$,
    \begin{equation*}
        \delta_1 \langle I \rangle = \delta_2 \langle I\rangle \implies R_I[\delta_1^{-1}] = R_I[\delta_2^{-1}].
    \end{equation*}
     For that, it suffices to show that if $g \in G_{\mathbb{R}}$, such that $g\langle I \rangle = I$, then $R_{I}[g^{-1}] = R_{I}$. Let us now write
    \begin{equation*}
        g = \m{\alpha & a_1 & x^{t} & a_{n+2} & \beta \\ b_1&A_{1,1}&E^{t}&A_{1,n+2}&c_1\\y&F&K&G&z\\b_{n+2}&A_{n+2,1}&H^{t}&A_{n+2,n+2}&c_{n+2}\\\gamma&d_1&w^{t}&d_{n+2}&\delta},
    \end{equation*}
    with $\alpha,\beta,\gamma,\delta, A_{1,1},A_{1,n+2},A_{n+2,1},A_{n+2,n+2} \in \mathbb{R}$, $E,F,G,H,x,y,z,w \in \mathbb{R}^{n}$ and $K \in \mathbb{R}^{n,n}$.
    By the definition of the action and the fact that $g \langle I \rangle = I$, we obtain (cf. \cite[(1.3)]{sugano}):
    \begin{equation*}
        g \m{1 \\ I \\ 1} = j(g, I) \m{1 \\ I \\ 1}.
    \end{equation*}
    But we have $j(g,I) = \gamma+\delta+(d_1+d_{n+2})i$. Hence, we obtain the following relations:
    \begin{empheq}[left={\empheqlbrace }]{align}
        &\alpha+\beta + i(a_1+a_{n+2}) = \gamma+\delta+i(d_1+d_{n+2}).\label{(1)}\\
        &b_1+c_1+i(A_{1,1}+A_{1,n+2}) = i(\gamma+\delta)-(d_1+d_{n+2}).\label{(2)}\\
        &y+z+i(F+G)=0.\label{(3)}\\
        &b_{n+2}+c_{n+2}+i(A_{n+2,1}+A_{n+2,n+2})=i(\gamma+\delta)-(d_1+d_{n+2}).\label{(4)}
    \end{empheq}
    Now, $g\langle I\rangle=I$ implies $g^{-1}\langle I\rangle=I$ as well. But, as in Lemma \ref{inverse}, we have
    \begin{equation*}
        g^{-1} = \m{\delta&c_{n+2}&-z^{t}S&c_1&\beta\\d_{n+2}&A_{n+2,n+2}&-G^{t}S&A_{1,n+2}&a_{n+2}\\-S^{-1}w&-S^{-1}H&S^{-1}K^{t}S&-S^{-1}E&-S^{-1}x\\d_1&A_{n+2,1}&-F^{t}S&A_{1,1}&a_1\\\gamma&b_{n+2}&-y^{t}S&b_1&\alpha},
    \end{equation*}
    and therefore, we also obtain the relations
    \begin{empheq}[left={\empheqlbrace }]{align}
        &\delta+ \beta +i(c_{n+2}+c_1) = \gamma+\alpha+i(b_1+b_{n+2}).\label{(5)}\\
        &a_{n+2}+d_{n+2} + i(A_{n+2,n+2}+A_{1,n+2})) = i(\gamma+\alpha)-(b_1+b_{n+2}).\label{(6)}\\
        &S^{-1}w+S^{-1}x + i(S^{-1}H+S^{-1}E)=0.\label{(7)}\\
        &a_1+d_1 +i(A_{n+2,1}+A_{1,1}) = i(\gamma+\alpha)-(b_1+b_{n+2}).\label{(8)}
    \end{empheq}
    Using now equations \eqref{(2)}, \eqref{(4)}, \eqref{(6)}, \eqref{(8)}, we obtain
    \begin{equation*}
        A_{1,1}=A_{n+2,n+2}, \textup{ } A_{1,n+2}=A_{n+2,1} \textup{ and }\alpha=\delta.
    \end{equation*}
    Using \eqref{(1)} we also get $\alpha+\beta= \gamma+\delta \implies \beta=\gamma$ as well. From \eqref{(3)}, \eqref{(7)} we get
    \begin{equation*}
        y=-z, \textup{ } w=-x, \textup{ } F=-G \textup{ and } H=-E.
    \end{equation*}
    Finally, from \eqref{(1)}, \eqref{(6)}, \eqref{(8)}, we have the equations
    \begin{equation*}
        a_1+d_1 = -(b_{n+2}+b_1), \ a_{n+2}+d_{n+2}=-(b_1+b_{n+2}), \ a_1+a_{n+2}=d_{1}+d_{n+2}.
    \end{equation*}
    These give $a_1+d_1=a_{n+2}+d_{n+2}$, which together with $ a_1+a_{n+2}=d_{1}+d_{n+2}$ gives
    \begin{equation*}
        a_1=d_{n+2} \textup{ and } a_{n+2}=d_1.
    \end{equation*}
    Similarly, $b_1=c_{n+2}$ and $c_1=b_{n+2}$ and these relations are enough to check $g^{t}R_I = R_{I}g^{-1}$, i.e. what we wanted to prove.
\end{proof}
We are now ready to give the proof of the main Proposition \ref{main_proposition_eisenstein}.
\begin{proof}
Let $M = \m{e_1 & e_2}$. We claim that the map
\begin{align*}
    \Gamma_S &\longrightarrow X\\
    \gamma & \longmapsto \gamma^{-1}M
\end{align*}
induces a bijection $\Gamma_{S,J} \backslash \Gamma_S \xrightarrow{\sim} X / \textup{GL}_2(\mathbb{Z})$.
\begin{itemize}
    \item \underline{Well-Defined}: Firstly, $\gamma^{-1}M \in X$ as the first two columns of $\gamma^{-1}$ are integer vectors and as $\gamma^{-1} \in \Gamma_S \subset \textup{GL}_{n+4}(\mathbb{Z})$, we get that $\gamma^{-1}M$ is primitive. Moreover, $S_1[\gamma^{-1}M]=0$, as we have already shown in Lemma \ref{lem: conditions}.\\
    Now, if $\delta = p\gamma$ for some $p \in \Gamma_{S,J}$, then we can write $p^{-1}M =  \m{a&b\\c&d\\\vdots&\vdots\\0&0}$ with $\m{a&b\\c&d} \in \textup{SL}_2(\mathbb{Z})$ (because of the form of the parabolic). Hence
    \begin{equation*}
        \delta^{-1}M = \gamma^{-1}p^{-1}M = \gamma^{-1}M\m{a&b\\c&d},
    \end{equation*}
    which gives the map is well-defined.
    \item \underline{Injective}: If $\gamma^{-1}M = \delta^{-1}MN$ for some $N \in \textup{GL}_2(\mathbb{Z})$, then
    \begin{equation*}
        \gamma\delta^{-1}M = MN^{-1},
    \end{equation*}
    which shows that $\gamma\delta^{-1} \in \Gamma_{S,J}$ (the first two columns of $\gamma\delta^{-1}$ belong in the $\mathbb{Z}$-span of the vectors $e_1,e_2$ and $\gamma, \delta \in \Gamma_{S} \subset \textup{SL}_{n+4}(\mathbb{Z})$).
    \item \underline{Surjective}: Consider now $\m{l&m} \in X/\textup{GL}_2(\mathbb{Z})$. The plane generated by the vectors $l, m$ can be checked to be isotropic. Indeed, 
    \begin{equation*}
        \textup{span}\{l,m\} = \textup{span}\left\{\frac{l}{\gcd{(S_1l)}}, \frac{m}{\gcd{(S_1m)}}\right\}
    \end{equation*}
    in $V_1$. These vectors both belong in $L_1^{*}$ (can be checked directly) and also $\gcd{\left(\frac{S_1l}{\gcd{(S_1l)}}\right)}=1$ and similarly for the other vector. The isotropy condition still holds for the two new vectors (as we only change them by a scalar), and therefore the claim follows.\\\\
    Now, because $\#\mathcal{C}^{1}(\Gamma_S)=1$ by assumption, if $U$ is the plane generated by the basis vectors $e_1,e_2$ ($U$ is an isotropic plane too), there is an element $K \in \Gamma_S$ such that $K$ maps $U$ to $W$. Hence, there exist $x,y,z,w \in \mathbb{Q}$ such that
    \begin{equation*}
        \m{Ke_1&Ke_2} = \m{l&m}\m{x&z\\y&w}.
    \end{equation*}
    But as $\m{l&m}$ is assumed to be primitive, we have by \cite[Lemma 3.3]{Shimura_Euler_Product} that $\exists A \in \textup{Mat}_{2,n+4}(\mathbb{Z})$ such that $A\m{l&m}=1_2$. Hence, we obtain
    \begin{equation*}
        A\m{Ke_1&Ke_2} = \m{x&z\\y&w},
    \end{equation*}
    and so $x,y,z,w\in \mathbb{Z}$. Now $\m{Ke_1&Ke_2}$ is also primitive by \cite[Lemma 3.3]{Shimura_Euler_Product}, as it can be completed to an element of $\textup{GL}_{n+4}(\mathbb{Z})$, namely $K$. Hence, $\exists B \in \textup{Mat}_{2,n+4}(\mathbb{Z})$ such that $B\m{Ke_1&Ke_2} = 1_2$. Hence,
    \begin{equation*}
        \m{x&z\\y&w}^{-1} = B\m{l&m} \in \textup{Mat}_{2}(\mathbb{Z}).
    \end{equation*}
    Hence, as the inverse of that matrix also has integer entries, we must have that its determinant is $\pm 1$, i.e., $xw-yz=\pm 1$.
    Therefore, $K^{-1}$ gets mapped to $\m{l & m}\textup{GL}_2(\mathbb{Z})$, as wanted.
\end{itemize}
The rest of the proof now follows from Proposition \ref{prop: majorants}.
\end{proof}
\begin{remark}\label{cusps_so(2,n)}

    We would like to make a few comments here regarding the condition $\#\mathcal{C}^{1}(\Gamma_S)=1$. From \cite[Theorem 5.10]{one-dimensional-cusps}, and because the lattice $L_1$ is maximal, we have that the condition $\#\mathcal{C}^{1}(\Gamma_S)=1$ is equivalent to $L$ (the definite lattice of rank $n$) having only one class in its genus, and $[\textup{O}(L):\textup{SO}(L)]=2$. For $n \geq 2$, a complete list of lattices with one class in their genus is known (see \cite{single-class-genera}, \cite{hanke2011enumerating} for example), and their rank is at most $10$. If $n=1$, $L=\mathbb{Z}$ also has a single class in its genus in our setting ($S=2t$, with $t \geq 1$ a square-free integer). Finally, for the condition $[\textup{O}(L):\textup{SO}(L)]=2$, a criterion we can use so that we ensure the existence of an element with determinant $-1$ (at least when $n$ is even as if $n$ is odd, $-1_n$ always works) can be found in \cite[Lemma 9.23, (iii)]{shimura2004arithmetic}. That is, the existence of some $x \in \mathbb{Z}^{n}$ so that $S[x]=2$. 
\end{remark}
Finally, we have the following Lemma, where we replace the condition of primitivity of the elements of $X$ in Proposition \ref{main_proposition_eisenstein} with the condition that the elements have maximal rank.
\begin{lemma}\label{maximal_rank}
    With the notation as above, we have
    \begin{equation*}
        \zeta(s)\zeta(s-1)E(Z,s) = \sum_{\substack{\ell \in \textup{Mat}_{n+4,2}(\mathbb{Z})/ \textup{GL}_{2}(\mathbb{Z})\\\textup{rank }\ell=2, S_1[\ell]=0}}\left(\det(R_{Z}[\ell])\right)^{-s/2}.
    \end{equation*}
\end{lemma}
\begin{proof}
    The proof is analogous to \cite[Lemma 3.1]{deitmar_krieg}. In particular, every matrix $\ell \in \textup{Mat}_{n+4,2}(\mathbb{Z})$ with $\textup{rank }\ell = 2$ and $S_1[\ell]=0$ can be written as $\ell = N\cdot M$ with $N$ being primitive and $S_1[N]=0$ and $M \in \textup{GL}_{2}(\mathbb{Q})\cap \textup{Mat}_{2,2}(\mathbb{Z})$. The proof then follows from the fact that
    \begin{equation*}
        \sum_{M \in (\textup{GL}_{2}(\mathbb{Q})\cap \textup{Mat}_{2,2}(\mathbb{Z}))/\textup{GL}_2(\mathbb{Z})} |\det(M)|^{-s} = \zeta(s)\zeta(s-1).
    \end{equation*}
    This can be found in \cite{deitmar_krieg} or its local version in \cite[Lemma 3.13]{Shimura_Euler_Product}.
\end{proof}
\section{Theta Series and Transformation Properties}\label{theta_series_section}
In order now to prove the analytic properties of the Klingen Eisenstein series $E(Z,s)$, we want to prove a theta correspondence between $\textup{SO}(2,n+2)$ and $\textup{Sp}_2$. That is, to integrate a Siegel Eisenstein series for $\textup{Sp}_2$ against an appropriately defined theta series and end up with $E(Z,s)$. However, as in most cases, the inner-product integrals will diverge, so we need to apply an appropriate differential operator first. In this Section, we recall the action of the symplectic group on Siegel's upper half plane, define the appropriate theta series and prove some important transformation properties.\\\\
The (real) symplectic group of degree $m \geq 1$ is defined by 
\begin{equation*}
    \textup{Sp}_m(\mathbb{R}) := \left\{g \in \textup{GL}_{2m}(\mathbb{R}) \mid g^{t}\m{0_m & -1_m\\1_m & 0_m}g = \m{0_m & -1_m\\1_m & 0_m}\right\}.
\end{equation*}
The Siegel's upper half plane is defined by 
\begin{equation*}
    \mathbb{H}_m := \left\{Z = X+iY \in M_m(\mathbb{C}) \mid X = X^{t}, Y = Y^{t}>0\right\}.
\end{equation*}
Now $g = \m{A&B\\C&D} \in \textup{Sp}_{m}(\mathbb{R})$ acts on $\mathbb{H}_m$ via
\begin{equation*}
    (g, Z) \longmapsto g \langle Z \rangle := (AZ+B)(CZ+D)^{-1}.
\end{equation*}
This defines a transitive action of $\textup{Sp}_m(\mathbb{R})$ on $\mathbb{H}_m$. We also define the factor of automorphy $j(g, Z) := \det(CZ+D)$. We call $\textup{Sp}_{m}(\mathbb{Z}) := \textup{Sp}_{m}(\mathbb{R}) \cap \textup{GL}_{2m}(\mathbb{Z})$ the full modular group. For any integer $N > 0$, we define the following congruence subgroup of $\textup{Sp}_m(\mathbb{Z})$:
\begin{equation*}
    \Gamma_0^{(m)}(N) := \left\{\m{A&B\\C&D} \in \textup{Sp}_{m}(\mathbb{Z}) \mid C \equiv 0 \pmod N\right\}.
\end{equation*}
We then have the following Definition (cf. \cite[Section 2.4.1]{Courtieu1991}):
\begin{defn}
    Let $k \in \mathbb{Z}$ and $m \geq 2$. A function $F : \mathbb{H}_m \longrightarrow \mathbb{C}$, with $F \in \mathcal{C}^{\infty}(\mathbb{H}_m)$, is called a $\mathcal{C}^{\infty}$-modular form of weight $k$ on the group $\Gamma_{0}^{(m)}(N)$ with a Dirichlet character $\psi \pmod N$, if for all
    \begin{equation*}
        \gamma = \m{A & B \\ C & D} \in \Gamma_{0}^{(m)}(N),
    \end{equation*}
    we have
    \begin{equation*}
        F(\gamma \langle Z \rangle) = \psi(\det D) \det(CZ+D)^{k}F(Z),
    \end{equation*}
    for all $Z \in \mathbb{H}_m$. We denote the space of such functions by $\widetilde{M}_k(N, \psi)$.
\end{defn}
Finally, in some cases, we can define a suitable inner product on the above space.
\begin{defn}\label{inner_product_congruence}
    The Petersson inner product of a pair of modular functions $F,G \in \widetilde{M}_k(N, \psi)$ is
    \begin{equation*}
        \left\langle F,G \right\rangle := \frac{1}{[\textup{Sp}_{m}(\mathbb{Z}):\Gamma_{0}^{(m)}(N)]} \int_{\Gamma_{0}^{(m)}(N)\backslash \mathbb{H}_m} F(Z)\overline{G(Z)}(\det Y)^{k}\textup{d}^{*}Z,
    \end{equation*}
    whenever the integral converges. Here, $\textup{d}^{*}Z$ is the $\textup{Sp}_{m}(\mathbb{R})$-invariant measure 
    \begin{equation*}
        (\det Y)^{-(m+1)}\textup{d}X\textup{d}Y.
    \end{equation*}
\end{defn}
We are now ready to give the definition of the theta series of interest, which is related to $\textup{Sp}_2$. Such theta series were first studied by Siegel.
\begin{defn}\label{theta_definition}
   Let $Z = X+iY \in \mathbb{H}_2$ and $W \in \mathcal{H}_S$. We then define the theta series\\
    \begin{equation*}
        \theta(Z,W) := \sum_{\ell \in \textup{Mat}_{n+4,2}(\mathbb{Z})}\theta_{\ell}(Z,W),
    \end{equation*}
    where for any $\ell \in \textup{Mat}_{n+4,2}(\mathbb{Z})$, we define
    \begin{equation*}
         \theta_{\ell}(Z,W) := e^{\pi i \textup{tr}(S_1[\ell]X)-\pi \textup{tr}(R_W[\ell]Y)} .
    \end{equation*}
    Let also $\Theta(Z,W) := \det(Y)^{\frac{n+2}{2}}\theta(Z,W)$. Because $R_W$ is positive definite, it follows that $\Theta$ converges absolutely and uniformly in any region of the form $\mathbb{H}_2(\epsilon) = \{Z = X+iY \mid Y \geq \epsilon 1_2\}$ with $\epsilon > 0$ and therefore defines a real analytic function of the matrices $X,Y$ with $X+iY \in \mathbb{H}_2$ (cf. \cite[page 291]{spherical_theta}). 
\end{defn}
We start with the following Lemma regarding the invariance of $\Theta$ under $\Gamma_S$. 
\begin{lemma}\label{transformation}
    Let $W \in \mathcal{H}_S$. We then have
    \begin{equation*}
        \Theta(Z, M\langle W\rangle) = \Theta(Z, W)
    \end{equation*}
    for all $M \in \Gamma_S$.
\end{lemma}
\begin{proof}
    Let $M \in \Gamma_S$. Then
    \begin{align*}
        \Theta(Z, M\langle W\rangle) &= \det(Y)^{\frac{n+2}{2}}\sum_{\ell \in \textup{Mat}_{n+4,2}(\mathbb{Z})} e^{\pi i \textup{tr}(S_1[\ell]X)-\pi \textup{tr}(R_{M\langle W\rangle}[\ell]Y)}\\ &= \det(Y)^{\frac{n+2}{2}}\sum_{\ell \in \textup{Mat}_{n+4,2}(\mathbb{Z})} e^{\pi i \textup{tr}(S_1[M^{-1}\ell]X)-\pi \textup{tr}(R_{W}[M^{-1}\ell]Y)}\\ &= \Theta(Z,W),
    \end{align*}
    because $R_{M\langle W \rangle} = R_{W}[M^{-1}]$ and $S_1[M^{-1}] = S_1$. These relations, together with the fact that $M \in \Gamma_{S} \subset \textup{GL}_{n+4}(\mathbb{Z}) \implies M^{-1} \in \textup{GL}_{n+4}(\mathbb{Z})$, gives the invariance under $\Gamma_{S}$.
\end{proof}
The situation, however, is quite different when it comes to the transformation of $\Theta$ with respect to $\textup{Sp}_{2}(\mathbb{Z})$. In particular, it is necessary to consider a specific congruence subgroup of $\textup{Sp}_2(\mathbb{Z})$ and even then $\Theta$ transforms as a modular form of a non-trivial weight and character.
\begin{proposition}[\cite{spherical_theta}, Theorem 1,2]\label{richter}
    Let $q$ denote the level of $S$, equivalently $S_1$ (see Definition \ref{level}). Then, for all $\gamma = \m{A & B \\ C &D} \in \Gamma_0^{(2)}(q)$, we have
    \begin{equation*}
        \Theta(\gamma Z, W) = \chi_{S}(\gamma)j(\gamma, Z)^{-n/2}\Theta(Z,W),
    \end{equation*}
    where $\chi_S(\gamma)$ is an eighth root of unity which does not depend on $Z$ and $W$ (cf. \cite[Theorem 1]{spherical_theta}).\\
    In the case when the rank $n$ of $S$ is even, we have $\chi_{S}(\gamma)=\psi_{S}(\det D)$, with $\psi_{S}$ a Dirichlet character modulo $q$ such that
    \begin{equation*}
        \psi_{S}(p) = \left(\frac{(-1)^{n/2}\det S}{p}\right)
    \end{equation*}
    for all odd primes $p$ (Legendre symbol) and $\psi_{S}(-1) = (-1)^{n/2}$ (cf. \cite[Theorem 2]{spherical_theta}). In particular, this implies $\Theta \in \widetilde{M}_k\left(q, \psi_{S}\right)$ with $k=-n/2$.
\end{proposition}
\section{Differential Operators}\label{diffrential_operators}
In this Section, we prepare the ground for the theta-correspondence. In the same fashion as Krieg in \cite{krieg1991}, Gritsenko in \cite{gritsenko}, and Raghavan and Sengupta in \cite{raghavan}, we need to apply some differential operators to the theta series first, so that the integral converges. Our first step is to make $\Theta$ invariant under the action of $\textup{Sp}_2(\mathbb{Z})$, up to the character $\chi_S$. This is essential because the differential operators that will eliminate the terms of the theta series that cause divergence are $\textup{Sp}_2(\mathbb{R})$-invariant. From now on, we assume that $4 \mid n$.\\

For $m \geq 1$, let $Z = X+iY \in \mathbb{H}_m$. We denote by $\partial_Z$ the matrix
\begin{equation*}
    (\partial_Z)_{ij} := \frac{1+\delta_{ij}}{2}\frac{\partial}{\partial Z_{ij}},
\end{equation*}
where $\delta_{ij}$ denotes the Kronecker's delta, for $1\leq i,j\leq m$. Here $\partial/\partial Z_{ij} := \left(\partial/\partial X_{ij}-i\partial/\partial Y_{ij}\right)/2$.
Due to Maass in \cite{maass}, we have the operator ($k \in \mathbb{Z}$)
\begin{equation}\label{maass shimura operator}
    \delta_{k} := \det(Y)^{-k+\frac{m-1}{2}}\det(\partial_Z)\det(Y)^{k-\frac{m-1}{2}},
\end{equation}
which sends functions from $\widetilde{M}_k(N,\psi_S)$ to $\widetilde{M}_{k+2}(N, \psi_S)$ (cf. \cite[Section 3.3.1]{Courtieu1991}).\\

For any integer $r \geq 1$, one then defines the Shimura differential operator as the composition
\begin{equation}\label{shimura operator}
    \delta_{k}^{(r)} := \delta_{k+2r-2}\cdots \delta_{k+2}\delta_{k}.
\end{equation}
This sends functions from $\widetilde{M}_k(N,\psi_S)$ to $\widetilde{M}_{k+2r}(N, \psi_S)$ (cf. \cite[Section 3.3.1]{Courtieu1991}). Therefore, in the case $m=2$, $\delta_k^{(r)}\Theta \in \widetilde{M}_0(q, \psi_S)$ with $k=-n/2$ and $r=-k/2 = n/4$, because of Proposition \ref{richter}. However, we still need to apply an invariant differential operator $R$, so that we remove the singular terms that cause the integrals to diverge. The existence of a suitable such operator is guaranteed by a result of Deitmar and Krieg in \cite{deitmar_krieg}. Let us now describe it here. \\

For any dimension $m \geq 1$, we consider the algebras $\mathbb{D}(\mathbb{H}_m)$ and $\mathbb{D}(\mathcal{P}_m)$ of invariant differential operators with respect to $\textup{Sp}_{m}(\mathbb{R})$ and $\textup{GL}_m(\mathbb{R})$. Here, $\mathcal{P}_m$ denotes the symmetric space $\textup{GL}_{m}(\mathbb{R})/\textup{O}_{m}(\mathbb{R})$ and can be identified with the set
\begin{equation*}
    \mathcal{P}_m = \left\{Y \in \textup{Mat}_{m}(\mathbb{R}) \mid Y = Y^{t}>0\right\},
\end{equation*}
by means of the action of $\textup{GL}_{m}(\mathbb{R})$ on $\mathcal{P}_m$ given by
\begin{equation*}
    M \cdot Y = Y[M^{t}] = MYM^{t}.
\end{equation*}
Next, if we consider the injective map
\begin{equation*}
    \phi : \mathcal{C}^{\infty}(\mathcal{P}_m) \longrightarrow \mathcal{C}^{\infty}(\mathbb{H}_m)
\end{equation*}
defined by
\begin{equation*}
    \phi(f)(X+iY) := f(Y),
\end{equation*}
we may associate to $\phi$ a well-defined map
\begin{equation*}
    \phi^{*} : \mathbb{D}(\mathbb{H}_m) \longrightarrow \mathbb{D}(\mathcal{P}_m),
\end{equation*}
which sends $D \longmapsto \phi^{-1} \circ D \circ \phi$. Then, from \cite[Theorem 1.1]{deitmar_krieg}, we obtain that $\phi^{*}$ is an injective algebra homomorphism.\\

Now, for any $l \in \mathbb{R}$, we  define the operator $D_{l} = D_{l,m} \in \mathbb{D}\left(\mathcal{P}_m\right)$ by
\begin{equation*}
    D_{l,m}(Y) := (\det Y)^{l}\det (\partial_Y) \det(Y)^{1-l}.
\end{equation*}
From \cite[Theorem 1.2]{deitmar_krieg}, we have that $D_{m+2-l,m}D_{l,m}$ belongs in the image of $\phi^{*}$. Define then 
\begin{equation*}
   R(m,l) := \left(\phi^{*}\right)^{-1}\left(D_{m+2-l,m}D_{l,m}\right) \in \mathbb{D}(\mathbb{H}_m).
\end{equation*}
In the case $m=2$, we write for $Z = X+ iY \in \mathbb{H}_2$,
\begin{equation*}
    Z = \m{z_1 & z_3\\z_3 & z_2}, \textup{ }X = \m{x_1 & x_3\\x_3 & x_2}, \textup{ } Y = \m{y_1 & y_3\\y_3 & y_2}.
\end{equation*}
We first have the following Lemma regarding the behaviour of $\Theta$ under the action of the Maass-Shimura operator.
\begin{lemma}\label{expression_under_delta}
Let $k=-n/2$ and $r = n/4$. For any $\ell \in \textup{Mat}_{n+4,2}(\mathbb{Z})$, we have
\begin{multline}\label{delta expansion}
    \delta_{k}^{(r)}\left[(\det Y)^{\frac{n+2}{2}}\theta_{\ell}(Z,W)\right] = (\det Y)^{1+r}p(\det((S_1+R_W)[\ell]Y), \textup{tr}((S_1+R_W)[\ell]Y))\theta_{\ell}(Z,W),
\end{multline}
where $p=p(U,V) \in \mathbb{R}[U,V]$ is a polynomial in two variables, and its coefficients do not depend on $\ell$. In particular, if $\ell \in \textup{Mat}_{n+4,2}(\mathbb{Z})$ is such that $S_1[\ell]=0$, this becomes  
\begin{equation}\label{p expression}
    \delta_{k}^{(r)}\left[(\det Y)^{\frac{n+2}{2}}\theta_{\ell}(Z,W)\right] = (\det Y)^{1+r}p\left(\det (R_{W}[\ell]Y), \textup{tr}\left(R_{W}[\ell]Y\right)\right)e^{-\pi\textup{tr}\left(R_{W}\left[\ell\right]Y\right)}.
\end{equation}
\end{lemma}
\begin{proof}
    For any $\ell \in \textup{Mat}_{n+4,2}(\mathbb{Z})$, we have by \cite[Lemma 1.1]{satoh1986differential}
    \begin{equation*}
        \delta_{k}^{(r)}\left[(\det Y)^{\frac{n+2}{2}}\theta_{\ell}(Z,W)\right] = (\det Y)^{\frac{n+2}{2}}\delta_{1}^{(r)}\theta_{\ell}(Z,W).
    \end{equation*}
    Define now the operator $\sigma \in \mathbb{D}(\mathbb{H}_2)$ by
    \begin{equation}\label{sigma operator}
        \sigma := i\sum_{j=1}^{3}y_{j}\frac{\partial}{\partial z_{j}}.
    \end{equation}
    By \cite[Proposition 1.2, (a)]{satoh1986differential}, $\delta_{1}^{(r)}\theta_{\ell}(Z,W)$ is a $\mathbb{Z}[1/2]$-linear combination of functions of the form
    \begin{equation}\label{satoh expression}
        \det(Y)^{-b}\sigma^{c}\left(\det(\partial_{Z})\right)^{d}\theta_{\ell}(Z,W),
    \end{equation}
    where $b, c, d$ are integers with $0\leq b\leq c\leq r, \ 0\leq d\leq r$ and $b+d=r$. Note that the additional requirements stated in \cite[Proposition 1.2]{satoh1986differential} are not needed for the proof of this part. We compute 
    \begin{multline}\label{a_l,b_l,c_l}
        \frac{\partial}{\partial z_{1}} \theta_{\ell}(Z,W) = A_{\ell}\theta_{\ell}(Z,W), \textup{ }\frac{\partial}{\partial z_{2}} \theta_{\ell}(Z,W) = B_{\ell}\theta_{\ell}(Z,W), \ \frac{\partial}{\partial z_{3}} \theta_{\ell}(Z,W) = 2C_{\ell}\theta_{\ell}(Z,W),
    \end{multline}
    where for $\ell = \m{l&m}$, with $l, m \in \mathbb{Z}^{n+4}$ we have 
    \begin{equation*}
        \displaystyle{A_{\ell} = \frac{1}{2}\pi i \left(S_1[l]+R_W[l]\right)}, \ \displaystyle{B_{\ell} = \frac{1}{2}\pi i \left(S_1[m]+R_W[m]\right)}, \ \displaystyle{C_{\ell} = \frac{1}{2}\pi i \left(l^{t}S_1m+l^{t}R_Wm\right)}.
    \end{equation*}
    We now set $\displaystyle{T_{\ell} := \frac{1}{2}\pi i (S_1 + R_{W})[\ell]}$ and write $T_{\ell} = \m{t_1 & t_3/2 \\ t_3/2 & t_2}$, with $t_1 = A_{\ell}, t_2 = B_{\ell}$ and $t_3 = 2C_{\ell}$.\\

    We first compute $\det (\partial_{Z})\theta_{\ell} = \det (T_{\ell})\theta_{\ell}$. Moreover, if $B := \textup{tr}(T_{\ell}Y)$, we claim that for any $0 \leq c \leq r$, $\sigma^{c}\theta_{\ell} = f_{c}(B)\theta_{\ell}$ for some polynomial $f$ of degree $c$. We show this by induction. Note that for $j=1,2,3$, we have $\partial B/\partial z_{j} = -it_j/2$. \\

    If $c = 0$, then the claim is clear with $f_0(B) = 1$. If now $\sigma^{c}\theta_{\ell} = f_{c}(B)\theta_{\ell}$, we have 
    \begin{equation*}
        \sigma^{c+1}\theta_{\ell} = i \sum_{j=1}^{3}y_j\left[f_{c}'(B)\left(\frac{-it_j}{2}\right)\theta_{\ell} + f_{c}(B)t_j\theta_{\ell}\right] = B\left(\frac{1}{2}f_{c}'(B)+if_{c}(B)\right)\theta_{\ell},
    \end{equation*}
    from which the claim follows with $\displaystyle{f_{c+1}(B) = B\left(\frac{1}{2}f_{c}'(B)+if_{c}(B)\right)}$ of degree $c+1$. 
    
    Now, each term in \eqref{satoh expression} can be written as
    \begin{multline*}
        (\det Y)^{-b}\sigma^{c}\det (\partial _{Z})^{d}\theta_\ell = (\det Y)^{-b}\sigma^{c}\det (T_\ell)^{d}\theta_\ell = (\det Y)^{-r}[\det (T_\ell Y)^{d}\sigma^{c}\theta_{\ell}] =\\= (\det Y)^{-r}[\det (T_\ell Y)^{d}f_{c}(\textup{tr}(T_\ell Y))\theta_{\ell}],
    \end{multline*}
    because $r - b = d$. Equation \eqref{delta expansion} now follows after absorbing the $\pi i /2$ factor of $T_{\ell}$ and observing (by induction) that $f_{c}(B)$ has purely imaginary coefficients in the odd powers of $B$ and real coefficients in the even powers. Hence $p$ will have real coefficients. Equation \eqref{p expression} now follows immediately from \eqref{delta expansion}. \qedhere 
    
\end{proof}

We are now ready to give the main Proposition of this Section, regarding the elimination of the terms that will cause the divergence of the integral. 
\begin{proposition}\label{divergent_terms}
Let $k=-n/2$ and $r=-k/2 = n/4$. Let also $R := R\left(2, 2+r\right)$. We then have
\begin{equation*}
    R\left[\delta_{k}^{(r)}\Theta\right](Z,W) = R\left[\sum_{\substack{\ell \in \textup{Mat}_{{n+4},2}(\mathbb{Z})\\\textup{rank } \ell=2}}\delta_{k}^{(r)} \left[(\det Y)^{\frac{n+2}{2}}e^{\pi i \textup{tr}(S_1[\ell]X+iR_{W}[\ell]Y)}\right]\right].
\end{equation*}
\end{proposition}

\begin{proof}
    Due to Maass in \cite{maass} and Shimura in \cite{shimura_differential}, $R$ can be written in the form 
    \begin{equation}\label{laplacians}
        R = u(H_1, H_2)
    \end{equation}
    where $H_1, H_2$ are some ``generalised'' Laplacians and $u \in \mathbb{C}[X,Y]$ a polynomial (see also \cite[Example 3.2]{yang} and \cite[Proposition 6]{bump_choie} for an explicit description of $H_1, H_2$). 
    Therefore, because $\Theta$ is absolutely and uniformly convergent on compact subsets of $\mathbb{H}_2$, we can apply the differential operators term by term. Hence, it suffices to show that for any $\ell \in \textup{Mat}_{n+4,2}(\mathbb{Z})$ with $\textup{rank } \ell < 2$, we have
    \begin{equation*}
        R\left[\delta_{k}^{(r)} \left[(\det Y)^{\frac{n+2}{2}}e^{\pi i \textup{tr}(S_1[\ell]X+iR_{W}[\ell]Y)}\right]\right] = 0.
    \end{equation*}
    Fix now $\ell \in \textup{Mat}_{n+4,2}(\mathbb{Z})$ such that $\textup{rank } \ell < 2$. We may assume that 
    \begin{equation*}
        S_1[\ell] = \m{x&0\\0&0}  \textup{ and } R_{W}[\ell] = \m{y&0\\0&0},
    \end{equation*}
    for some $x,y \in \mathbb{R}$.
    This is true because we can find $U \in \textup{SL}_2(\mathbb{R})$ such that $\ell U = \m{a&0}$ for some $a \in \mathbb{R}^{n+4}$. But then
    \begin{equation*}
        \widetilde{U} := \m{U&0\\0&(U^{t})^{-1}} \in \textup{Sp}_{2}(\mathbb{R})    
    \end{equation*}
    and the action of this matrix on $\mathbb{H}_2$ is $Z \longmapsto UZU^{t}$. Since $R$ is $\textup{Sp}_2$-invariant and $j(\widetilde{U}, Z)=1$, $\psi_S(\det U)=1$ (so the action $Z \longmapsto \widetilde{U}\langle Z \rangle$ does not change $\delta_k^{(r)}\Theta$), we can change variables $Z \longmapsto UZU^{t}$. But $\det(UYU^t) = \det(Y)$ and
    \begin{equation*}
        \textup{tr}(S_1[\ell]UXU^{t}) = \textup{tr}(U^{t}S_1[\ell]UX) = \textup{tr}(S_1[\ell U]X),
    \end{equation*}
    so we can replace $\ell$ with $\ell U$, which will then give the form of $S_1[\ell]$ wanted. Similarly for $R_{W}[\ell]$.
    By \cite[Lemma 1.1]{satoh1986differential}, we have (we remind here that $k=-n/2$)
    \begin{equation*}
        \delta_{k}^{(r)}\left[(\det Y)^{\frac{n+2}{2}}\theta_{\ell}(Z,W)\right] = (\det Y)^{\frac{n+2}{2}}\delta_{1}^{(r)}\left[\theta_{\ell}(Z,W)\right]
    \end{equation*}
    and by \cite[Proposition 1.2, (a)]{satoh1986differential}, we have that the quantity $\delta_{1}^{(r)}\theta_{\ell}(Z,W)$ will be a $\mathbb{Z}[1/2]$-linear combination of functions of the form
    \begin{equation*}
        \det(Y)^{-b}\sigma^{c}(\det(\partial_Z))^{d}\theta_{\ell}(Z,W),
    \end{equation*}
    where $b,c,d$ are integers with $0\leq b\leq c\leq r, \ 0\leq d\leq r$ and $b+d=r$ (the operator $\sigma$ here is as in \eqref{sigma operator}). Now, in the case $\textup{rank }\ell <2$, we have that $B_{\ell}=C_{\ell}=0$ in \eqref{a_l,b_l,c_l} because of the form of $S_1[\ell]$ and $R_{W}[\ell]$. Hence, 
    \begin{equation*}
        \det(\partial_{Z})\left[\theta_{\ell}(Z,W)\right] = 0.
    \end{equation*}
    So, we only need to consider the case $b=c=r,\ d=0$. But from the proof of Lemma \ref{expression_under_delta}, we have $\sigma^{r}\theta_{\ell}(Z,W) = f_{r}(B)\theta_{\ell}(Z,W)$, with $B = \textup{tr}(T_{\ell}Y)$ and $T_{\ell}$ as in Lemma \ref{expression_under_delta}. But, since $B_\ell = C_{\ell}=0$, we have that $B$, hence $\sigma^{r}\theta_{\ell}$ depends only on $y_1$ and $\theta_{\ell}$.
Consider now the map
\begin{equation*}
    \tau : \mathcal{C}^{\infty}(\mathbb{H}_1 \times \mathbb{H}_1) \longrightarrow \mathcal{C}^{\infty}(\mathbb{H}_2),
\end{equation*}
defined by
\begin{equation*}
    \tau(h)\left(\m{Z_1 & a \\ a & Z_2}\left[\m{1&b\\0&1}\right]\right) = h(Z_1, Z_2) \textup{ } \forall a,b \in \mathbb{R}.
\end{equation*}
Now, if we write 
\begin{equation*}
    Z = \m{Z_1 & a \\ a & Z_2}\left[\m{1&b\\0&1}\right] = X+iY,
\end{equation*}
we observe that 
\begin{equation*}
    Y = \m{Y_1&Y_1b\\bY_1&b^2Y_1+Y_2},
\end{equation*}
where $Z_j=X_j+iY_j, \ j=1,2$. Hence, $\det(Y) = Y_1Y_2$ and $\sigma^{r}\theta_{\ell}(Z,W)$ depends only on $\theta_{\ell}(Z,W)$ and $Y_1$. But
\begin{equation*}
    \theta_{\ell}(Z,W) = e\left(\frac{1}{2}\left(xX_1+iyY_1\right)\right),
\end{equation*}
where $S_1[\ell] = \m{x&0\\0&0}$ and $R_{W}[\ell] = \m{y&0\\0&0}$.
Therefore,
\begin{equation*}
    \delta_{k}^{(r)}\left[(\det Y)^{\frac{n+2}{2}}\theta_{\ell}\right](Z,W)
\end{equation*}
is independent of $a,b$ and so belongs in the image of $\tau$. But by \cite[Proposition 1.1]{deitmar_krieg}, $\tau^{-1} \circ R \circ \tau$ is a simple tensor in $\mathbb{D}(\mathbb{H}_1\times \mathbb{H}_1)$, so the problem is reduced in the one-dimensional case and $S_1[\ell] = R_W[\ell] = 0$, i.e., suffices to show
\begin{equation*}
    R(1,2+n/4)[\delta_{k}^{(r)}y^{\frac{n+2}{2}}] = 0,
\end{equation*}
where now we have $\displaystyle{\delta_{k} = \frac{k}{2iy}+\frac{\partial}{\partial z}}$ in the one-dimensional case.
But we have
\begin{equation*}
    \delta_{k}^{(r)}y^{\frac{n+2}{2}} = \textup{(const)} \times y^{n/4+1}.
\end{equation*}
Also, from \cite[page 807]{bocherer}, we have that in the one-dimensional case
\begin{equation*}
    R(1,2+n/4) = 4y^2\frac{\partial}{\partial{z}}\frac{\partial}{\partial{\overline{z}}} - \frac{n}{4}\left(\frac{n}{4}+1\right)
\end{equation*}
(note that $R(m, k) = R^{\textup{B{\"o}ch}}(m, k-1)$, where $R^{\textup{B{\"o}ch}}$ is the operator defined in \cite[Theorem 2.1]{bocherer}). Therefore, the result follows.
\end{proof}
Finally, combining the two above results, we obtain the following Lemma.
\begin{lemma}\label{bound}
    Let $k=-n/2$ and $r=n/4$, as before. Given $l \in \mathbb{R}\textup{, } \epsilon >0$, and a compact subset $\mathcal{C}$ of $\mathcal{H}_S$, there exists a constant $C>0$ such that 
\begin{equation*}
    \left|R\left[\delta_{k}^{(r)}\Theta\right](Z,W)\right| \leq C(\det Y)^{l}
\end{equation*}
holds for all $W \in \mathcal{C}$ and $Z=X+iY \in \mathbb{H}_2$, with $Y \geq \epsilon 1_2$.
\end{lemma}
\begin{proof}
    From Proposition \ref{divergent_terms}, the fact that $R$ is well-behaved (see \eqref{laplacians}), \eqref{satoh expression} and the relations we have obtained in the proof of Lemma \ref{expression_under_delta}, we can write
    \begin{equation}\label{depends only on Y}
        R\left[\delta_{k}^{(r)}\Theta\right](Z,W) = (\det Y)^{\frac{n+2}{2}}\sum_{\substack{\ell \in \textup{Mat}_{n+4,2}(\mathbb{Z})\\ \textup{rank }\ell=2} }g(S_1[\ell], R_{W}[\ell], Y)\theta_{\ell}(Z,W),
    \end{equation}
    for some polynomial $g$ in the entries of $S_1[\ell], R_W[\ell]$, and $Y$. The rest of the proof now follows in the same way as in \cite[Proposition 2.1, (b)]{deitmar_krieg}.
\end{proof}
\section{Theta Correspondence}\label{theta-correspondence}
In this Section, we finally give the theta correspondence between the Klingen-type Eisenstein series of $\textup{SO}(2,n+2)$ and the Siegel-type Eisenstein series for $\textup{Sp}_2$. We start with the following definition.
\begin{defn}\label{siegel_eisenstein}
    Let $\chi=\chi_S$ denote the character of Proposition \ref{richter}. Let also
    \begin{equation*}
        P_{2,0} := \left\{\m{A&B\\C&D} \in \textup{Sp}_{2}(\mathbb{Z}) \mid C=0\right\}
    \end{equation*}
    denote the Siegel parabolic subgroup of $\textup{Sp}_2$. Notice that $P_{2,0} \cap \Gamma_0^{(2)}(q) = P_{2,0}$. For $s \in \mathbb{C}$ with $\textup{Re}(s)>3/2$, we then define the Siegel Eisenstein series with respect to $P_{2,0}$ and with character $\chi$ as
    \begin{equation}\label{symplectic eisenstein}
        \widetilde{E}(Z,\chi, s) := \sum_{\gamma \in P_{2,0} \backslash \Gamma_{0}^{(2)}(q)}\overline{\chi(\gamma)}(\det(\textup{Im}(\gamma Z)))^{s}.
    \end{equation}
\end{defn}
If $\delta=\m{A&B\\0&D} \in P_{2,0}$, we have $\det D = \pm 1$, so $\chi(\delta) = \psi(\det D) = 1$ by Proposition \ref{richter} (recall that $4 \mid n$). Therefore, the Eisenstein series is well-defined. This Eisenstein series has a meromorphic continuation to $\mathbb{C}$ (see \cite[Theorem 1, (3)]{kalinin}). Finally, using the fact that $\chi$ is a character, we have
\begin{equation*}
    \widetilde{E}(\gamma Z, \chi, s) = \chi(\gamma)\widetilde{E}(Z,\chi, s),
\end{equation*}
for all $\gamma \in \Gamma_{0}^{(2)}(q)$.
The main Theorem of the paper can now be stated as follows:
\begin{theorem}\label{main theorem}
    Let $S$ have rank $n$, with $4 \mid n$ and be such that $\#\mathcal{C}^{1}(\Gamma_S) = 1$. Let also $k=-n/2$ and $r=n/4$, as before. Define
    \begin{equation*}
        \phi_2(s) := s\left(s-\frac{1}{2}\right), \textup{ } \xi(s):=\pi^{-s/2}\Gamma(s/2)\zeta(s),
    \end{equation*}
    for $s \in \mathbb{C}$. We then have for $\hbox{Re}(s)>n+1$
    \begin{equation*}
        \left\langle \widetilde{E}(Z, \chi, (s+1)/2-r), R[\delta_{k}^{(r)}\Theta](Z,W)\right \rangle_{\Gamma_{0}^{(2)}(q)} = \xi(s)\xi(s-1)\gamma_S(s)E(W,s),
    \end{equation*}
    where $\displaystyle{\gamma_S(s) := \frac{1}{[\textup{Sp}_{2}(\mathbb{Z}):\Gamma_{0}^{(2)}(q)]}(-4)^{-r}\phi_2(s/2-2r)\phi_2(s/2)\prod_{j=1}^{r}\phi_2\left(\frac{s-(2j-1)}{2}\right)}$.
\end{theorem}
\begin{proof}
Let 
\begin{equation*}
    I:= [\textup{Sp}_{2}(\mathbb{Z}):\Gamma_{0}^{(2)}(q)]\left\langle \widetilde{E}(Z, \chi, (s+1)/2-r), R[\delta_{k}^{(r)}\Theta](Z,W)\right\rangle_{\Gamma_{0}^{(2)}(q)}.
\end{equation*}
First of all, this integral is well-defined because of Lemma \ref{bound}. We then have
\begin{align*}
    I &= \bigintss_{\Gamma_{0}^{(2)}(q)\backslash \mathbb{H}_2} \left[\sum_{\gamma \in P_{2,0}\backslash \Gamma_{0}^{(2)}(q)}\overline{\chi(\gamma)}(\det(\textup{Im}(\gamma Z)))^{(s+1)/2-r}\overline{R[\delta_{k}^{(r)}\Theta](Z,W)}\right]\hbox{d}^{*}Z\\
    &= \bigintss_{\Gamma_{0}^{(2)}(q)\backslash \mathbb{H}_2} \left[\sum_{\gamma \in P_{2,0}\backslash \Gamma_{0}^{(2)}(q)}\overline{\chi(\gamma)}(\det(\textup{Im}(\gamma Z)))^{(s+1)/2-r}\chi(\gamma)\overline{R[\delta_{k}^{(r)}\Theta](\gamma Z, W)}\right]\hbox{d}^{*}Z\\
    &= \int_{P_{2,0}\backslash \mathbb{H}_2} (\det Y)^{(s+1)/2-r}\overline{R[\delta_{k}^{(r)}\Theta](Z,W)}\hbox{d}^{*}Z \\&=\int_{\mathcal{C}(2,\mathbb{R})}\int_{\mathcal{R}(2,\mathbb{R})} (\det Y)^{(s+1)/2-r}\overline{R[\delta_{k}^{(r)}\Theta](Z,W)}\hbox{d}^{*}Z,
\end{align*}
where in the second equation we used the invariance up to $\chi$ of $R[\delta_k^{(r)}\Theta]$ and in the third equation we used the usual unfolding trick. Here, $C(2, \mathbb{R})+i\mathcal{R}(2,\mathbb{R})$ is a fundamental domain for the action of $P_{2,0}$ on $\mathbb{H}_2$, where $C(2, \mathbb{R})$ denotes a fundamental parralepiped of $\textup{Sym}(2,\mathbb{Z})$ in $\textup{Sym}(2, \mathbb{R})$ and $\mathcal{R}(2,\mathbb{R})$ the Minkowski reduced matrices, as in \cite[p. 29]{Krieg_Book}. In the following, we write $\textup{Sym}_2(\mathbb{R}/\mathbb{Z})$ for $\mathcal{C}(2,\mathbb{R})$. Now, from Proposition \ref{divergent_terms}, Lemma \ref{expression_under_delta} and the proof of Lemma \ref{bound}, we have
\begin{multline*}
    \int_{\textup{Sym}_{2}(\mathbb{R}/\mathbb{Z})}\overline{R[\delta_{k}^{(r)}\Theta](Z,W)}\hbox{d}X =\\ =\sum_{\substack{\ell \in \textup{Mat}_{n+4,2}(\mathbb{Z})\\\textup{rank }\ell=2, S_1[\ell]=0}} R_0^{*}\left[p\left(\det (R_{W}[\ell]Y), \textup{tr}\left(R_{W}[\ell]Y\right)\right)e^{-\pi\textup{tr}\left(R_{W}[\ell]Y\right)}\right],
\end{multline*}
where $R_{0}^{*} = (\det Y)^{2-r}\det(\partial_{Y})(\det Y)^{1+2r}\det (\partial_Y)$. The fact that only terms with $S_1[\ell]=0$ remain follows from \eqref{depends only on Y}. The rest of the expression of this integral follows from \eqref{p expression} and from the fact that $R = (\phi^{*})^{-1}\left(R_{0}^{*}(\det Y)^{-(1+r)}\right)$ (note that from Lemma \ref{expression_under_delta}, the polynomial $p$ has real coefficients). Hence, from \cite[Proposition I.4.4, I.4.5]{Krieg_Book}, the integral $I$ becomes
\begin{equation*}
    \sum_{\substack{\ell \in \textup{Mat}_{n+4,2}(\mathbb{Z})/ \textup{GL}_{2}(\mathbb{Z})\\\textup{rank}\ell=2, S_1[\ell]=0}}\int_{\mathcal{P}_{2}} (\det Y)^{s/2}R_0\left[p\left(\det (R_{W}[\ell]Y), \textup{tr}\left(R_{W}[\ell]Y\right)\right)e^{-\pi\textup{tr}\left(R_{W}[\ell]Y\right)}\right]\hbox{d}^{*}Y,
\end{equation*}
where now $R_0 = (\det Y)^{-(1+r)}R_0^{*}$ and $\hbox{d}^{*}Y = (\det Y)^{-\frac{3}{2}}\hbox{d}Y$ is the invariant measure for the action of $\textup{GL}_2$ on $\mathcal{P}_{2}$. Now, for every $\ell$ in our sum, we have that $R_{W}[\ell]$, hence $(R_{W}[\ell])^{-1}$, is positive definite. So, we can write $(R_{W}[\ell])^{-1} = AA^{t}$ with $A$ lower triangular and change variables $Y \longmapsto Y[A^{t}]$. Then $\det (R_{W}[\ell]Y)$ and $\textup{tr}\left(R_{W}[\ell]Y\right)$ become $\det Y$ and $\textup{tr }Y$ respectively and the measure remains invariant. Now, from \cite[equation (2.4)]{bocherer}, $R_0$ is $\textup{GL}_{2}$-invariant ($n=2$ and $m=2+2r$ in the notation there). Hence, the integral becomes
\begin{equation*}
    I = \sum_{\substack{\ell \in \textup{Mat}_{n+4,2}(\mathbb{Z})/ \textup{GL}_{2}(\mathbb{Z})\\\textup{rank }\ell=2, \ S_1[\ell]=0}}(\det R_{W}[\ell])^{-s/2}\int_{\mathcal{P}_2}(\det Y)^{s/2} R_0[p(\det Y, \textup{tr}Y)e^{-\pi \textup{tr}Y}]\hbox{d}^{*}Y.
\end{equation*}
Let now $M = \det Y \det \partial_{Y}$. For any $t \in \mathbb{R}$, we have for its adjoint operator $\widehat{M}$ (see \cite[page 57]{maass} for a definition) that
\begin{equation*}
    \widehat{M}[(\det Y)^{t}] = \phi_{2}(t)(\det Y)^{t},
\end{equation*}
where $\displaystyle{\phi_{2}(t) = t\left(t-\frac{1}{2}\right)}$ (see \cite[(3.3)]{bocherer}). Therefore, we have
\begin{align*}
    I &= \!\begin{multlined}[t][10.5cm]\sum_{\substack{\ell \in \textup{Mat}_{n+4,2}(\mathbb{Z})/ \textup{GL}_{2}(\mathbb{Z})\\\textup{rank }\ell=2, S_1[\ell]=0}}(\det R_{W}[\ell])^{-s/2}\times\\\times\int_{\mathcal{P}_2}(\det Y)^{s/2-2r}M\left[(\det Y)^{1+2r}\det (\partial_{Y})\left[p(\det Y, \textup{tr}Y)e^{-\pi \textup{tr}Y}\right]\right]\hbox{d}^{*}Y\end{multlined}\\
    &=\sum_{\substack{\ell \in \textup{Mat}_{n+4,2}(\mathbb{Z})/ \textup{GL}_{2}(\mathbb{Z})\\\textup{rank }\ell=2, \ S_1[\ell]=0}}(\det R_{W}[\ell])^{-s/2}\phi_{2}(s/2-2r)\int_{\mathcal{P}_2}(\det Y)^{s/2}M[p(\det Y, \textup{tr}Y)e^{-\pi \textup{tr}Y}]\hbox{d}^{*}Y\\
    &=\sum_{\substack{\ell \in \textup{Mat}_{n+4,2}(\mathbb{Z})/ \textup{GL}_{2}(\mathbb{Z})\\\textup{rank }\ell=2, \ S_1[\ell]=0}}(\det R_{W}[\ell])^{-s/2}\phi_{2}(s/2-2r)\phi_{2}(s/2)\int_{\mathcal{P}_2}(\det Y)^{s/2}p(\det Y, \textup{tr}Y)e^{-\pi \textup{tr}Y}\hbox{d}^{*}Y\\
    &=\zeta(s)\zeta(s-1)E(W,s)\phi_{2}(s/2-2r)\phi_{2}(s/2)\int_{\mathcal{P}_2}(\det Y)^{s/2}p(\det Y, \textup{tr}Y)e^{-\pi \textup{tr}Y}\hbox{d}^{*}Y,
\end{align*}
where the second and third equality follow after transferring $M$ to its adjoint. The last equality follows from Proposition \ref{main_proposition_eisenstein} and Lemma \ref{maximal_rank}. We now just need to compute the last integral. It is true that (cf. \cite[page 80, 81]{maass})
\begin{equation*}
    \int_{\mathcal{P}_2} (\det Y)^{s}e^{-\textup{tr}(TY)}\hbox{d}^{*}Y = \pi^{1/2}\Gamma_{2}(s)(\det T)^{-s},
\end{equation*}
for any $s \in \mathbb{C}$ with $\textup{Re}(s) >1/2$ and $T \in \mathcal{P}_2$. Here $\Gamma_2(s) := \Gamma(s)\Gamma\left(s-1/2\right)$. Setting now $T \longmapsto \pi T$ and then applying the operator $\delta_1^{(r)}(S,T)$, where $S$ is symmetric, so that $U:= S+iT \in \mathbb{H}_2$, we obtain from Lemma \ref{expression_under_delta} and \eqref{p expression}
\begin{multline*}
    \int_{\mathcal{P}_2} (\det Y)^{s/2}(\det T)^{-r}p(\det (YT), \textup{ tr}(YT))e^{-\pi\textup{tr}(YT)}\hbox{d}^{*}Y = \pi^{1/2-s}\Gamma_{2}(s/2)\delta_1^{(r)}[(\det T)^{-s/2}].
\end{multline*}
But for any $\alpha \in \mathbb{Z}$ and $w \in \mathbb{C}$, we have from \eqref{maass shimura operator}
\begin{align*}
    \delta_{\alpha} \left[(\det T)^{w}\right] &= (\det T)^{1/2-\alpha}\det(\partial_U)(\det T)^{\alpha-1/2+w} \\
    &= -\frac{1}{4}(\det T)^{1/2-\alpha}\det(\partial_T)(\det T)^{\alpha-1/2+w}\\
    &= -\frac{1}{4}(\det T)^{1/2-\alpha} (\alpha-1/2+w)(\alpha+w)(\det T)^{\alpha-3/2+w}\\
    &= -\frac{1}{4}\phi_2(w+\alpha)(\det T)^{w-1},
\end{align*}
because $\det (\partial_{U})B(T) = -\det (\partial_{T})B(T)/4$ for a function $B=B(T)$ depending only on $T$ and the third equality follows from a well-known formula (see \cite[Theorem 2.2]{cayley_identity} for example). Hence, by successively applying the above and using the fact that $\phi_2(s) = \phi_2(1/2-s)$ for any $s \in \mathbb{C}$, we get
\begin{equation*}
    \delta_{1}^{(r)} \left[(\det T)^{-s/2}\right] = (-4)^{-r}\prod_{j=1}^{r}\phi_2\left(\frac{s-(2j-1)}{2}\right)(\det T)^{-s/2-r}.
\end{equation*}
By evaluating at $T = 1_2$, the proof is complete.
\end{proof}
\begin{corollary}\label{poles of eisenstein}
    Assume $S$ satisfies the assumptions of Theorem \ref{main theorem}. Then, $E(W,s)$ admits a meromorphic continuation to the complex plane and 
    \begin{equation*}
        \Gamma\left(\frac{s+1}{2}-r\right)\Gamma(s-2r)L_q(s+1-2r, \chi)\zeta_q(2s-4r)\xi(s)\xi(s-1)\gamma_S(s)E(W,s)
    \end{equation*}
    has only possible simple poles at $s\in \{(n+2)/2, \ (n+4)/2\}$ if $q \neq 1$ and at $s\in \{(n-2)/2, \ n/2, \ (n+2)/2,\  (n+4)/2\}$ if $q=1$. Here, $L(s,\chi)$ denotes the usual Dirichlet $L$-function attached to $\chi$, and the subscript $_{q}$ means that we omit the Euler factors sharing prime factors with $q$.
\end{corollary}
\begin{proof}
    From \cite[Theorem 19.3]{Shimura_Euler_Product}, we have that
    \begin{equation*}
        \Gamma(s)\Gamma(2s-1)L_{q}(2s, \chi)\zeta_q(4s-2)\widetilde{E}(Z, \chi, s)
    \end{equation*}
    has a meromorphic continuation to the complex plane with possible simple poles only at $s\in \{1,3/2\}$ if $q \neq 1$ and at $s\in\{0,1/2,1,3/2\}$ if $q=1$ (see \cite[(19.3.1), (19.3.2)]{Shimura_Euler_Product} because the character $\chi$ satisfies $\chi^2=1$). Hence, the Corollary follows from Theorem \ref{main theorem}.
\end{proof}
\begin{corollary}\label{corollary}
    From Proposition \ref{integral_representation_dirichlet} and Corollary \ref{poles of eisenstein}, we obtain the meromorphic continuation of $\mathcal{D}_{F,G}(s)$ to $\mathbb{C}$, due to the one of $\widetilde{E}(Z,\chi, s)$, as we noted in the beginning of Section \ref{theta-correspondence}.
\end{corollary}
\begin{remark}
    The conditions of Theorem \ref{main theorem} are satisfied when $S$ corresponds to the matrices defining the quadratic forms in $4$ and $8$ variables listed in the tables of \cite[Section 8]{hanke2011enumerating}. These have one class in the genus of maximal lattices, and by looking at their Grammian matrix, if $x$ is the first standard basis vector, we have $S[x]=2$. Hence, from Remark \ref{cusps_so(2,n)}, the condition $\#\mathcal{C}^{1}(\Gamma_S)=1$ is true in this setting.
\end{remark}
\section{The $E_8$ lattice}\label{e_8-section}
As an application, we obtain a precise result regarding the functional equation of $\mathcal{D}_{F,G}(s)$ in the case when
\begin{equation*}
    S = \m{2&-1&0&0&0&0&0&0\\-1&2&-1&0&0&0&0&0\\0&-1&2&-1&0&0&0&0\\0&0&-1&2&-1&0&0&0\\0&0&0&-1&2&-1&0&-1\\0&0&0&0&-1&2&-1&0\\0&0&0&0&0&-1&2&0\\0&0&0&0&-1&0&0&2}.
\end{equation*}
This is a positive definite, even matrix with $\det S = 1$ and the lattice it corresponds to is the so-called $E_8$ lattice (cf. \cite[Example 1.2.10]{eisenstein_thesis}). This is the unique unimodular lattice with only one-dimensional cusp (\cite[Example 1.6.20]{eisenstein_thesis}). As this is unimodular, the level $q$ is $1$ and so the character $\chi_S$ of Proposition \ref{richter} is trivial. Therefore, in this case, $\Gamma_0^{(2)}(q)$ is the whole $\textup{Sp}_{2}(\mathbb{Z})$ and the symplectic Eisenstein series of \eqref{symplectic eisenstein} is
\begin{equation*}
    \widetilde{E}(Z,s) = \sum_{\gamma \in P_{2,0}\backslash \textup{Sp}_2(\mathbb{Z})} (\det \textup{Im}(\gamma Z))^{s}.
\end{equation*}
Now, if $\xi(s) = \pi^{-s/2}\Gamma(s/2)\zeta(s)$, from \cite[Theorem 19.3]{Shimura_Euler_Product} (as in Corollary \ref{poles of eisenstein}), the modified Eisenstein series
\begin{equation*}
    \widetilde{\mathcal{E}}(Z,s):=\xi(2s)\xi(4s-2)\widetilde{E}(Z,s)
\end{equation*}
has a meromorphic continuation to $\mathbb{C}$ with possible simple poles only at $s \in \{0, 1/2, 1, 3/2\}$, and from \cite[Theorem 2]{kalinin} satisfies the functional equation
\begin{equation*}
    \widetilde{\mathcal{E}}\left(Z,3/2-s\right) = \widetilde{\mathcal{E}}(Z,s).
\end{equation*}
Now, for $W \in \mathcal{H}_S$, let
\begin{equation*}
    E^{*}(W,s):=\xi(s-3)\xi(2s-8)\xi(s)\xi(s-1)\gamma_S(s)E(W,s).
\end{equation*}
From our main Theorem \ref{main theorem}, we have in this case
\begin{equation*}
    \left\langle \widetilde{E}(Z, (s-3)/2), R[\delta_{-4}^{(2)}\Theta](Z,W)\right \rangle= \xi(s)\xi(s-1)\gamma_S(s)E(W,s).
\end{equation*}
Hence, we obtain that $E^{*}(W,s)$ has a meromorphic continuation to $\mathbb{C}$ and is invariant under $s \longmapsto 9-s$. Now, from Proposition \ref{integral_representation_dirichlet}, we have
\begin{equation*}
    (4\pi)^{-s}\Gamma(s)\mathcal{D}_{F,G}(s) = \#\textup{SO}(S;\mathbb{Z})\cdot \langle F(W)\cdot E(W,s-k+9)\textup{, } G(W)\rangle.
\end{equation*}
Hence, if we define
\begin{equation*}
    \mathcal{D}_{F,G}^{*}(s) := (4\pi)^{-s}\Gamma(s)\xi(s-k+6)\xi(2s-2k+10)\xi(s-k+9)\xi(s-k+8)\gamma_S(s-k+9)\mathcal{D}_{F,G}(s),
\end{equation*}
we have
\begin{equation*}
    \mathcal{D}_{F,G}^{*}(s) = \#\textup{SO}(S;\mathbb{Z}) \cdot \langle F(W)\cdot E^{*}(W,s-k+9)\textup{, } G(W)\rangle.
\end{equation*}
\\Therefore, we arrive at the following Theorem.
\begin{theorem}\label{e_8}
    Let $S$ be as above, corresponding to the $E_8$ lattice. With the notation as above, $\mathcal{D}_{F,G}^{*}(s)$ has a meromorphic continuation to $\mathbb{C}$ and is invariant under $s \longmapsto 2k-9-s$.
\end{theorem}
\textbf{Data Availability}: Data sharing not applicable to this article as no datasets were generated or analysed during the current study.
\vspace{-0.1cm}
\printbibliography

@phdthesis{ajouz,
author       = {Ajouz, A.},
title        = {{H}ecke {O}perators on {J}acobi {F}orms of {L}attice {I}ndex and the {R}elation to {E}lliptic {M}odular
{F}orms.},
school       = {University of Siegen},
type         = {PhD Thesis},
year         = {2015},
url = {https://dspace.ub.uni-siegen.de/handle/ubsi/938}
}

@article{spherical_theta,
year = {1989},
publisher = {},
volume = {\textbf{62}},
number = {\textbf{2}},
pages = {289},
author = {A. N. Andrianov},
title = {Spherical Theta Series},
journal = {Mathematics of the USSR-Sbornik},
}

@article{bump_choie,
author = {Bump, D. and Choie, Y.},
year = {2006},
pages = {},
title = {Derivatives of Modular Forms of Negative Weight},
volume = {\textbf{2}},
number = {\textbf{1}},
journal = {Pure and Applied Mathematics Quarterly},
}

@article{bocherer,
author = {Böcherer, S. and F. L. Chiera},
year = {2008},
pages = {801-824},
title = {On Dirichlet series and Petersson products for Siegel modular forms},
volume = {\textbf{58}},
number = {\textbf{3}},
journal = {Annales de l’Institut Fourier},
}

@book{Borcherds_Products,
    title = {Borcherds Products on $O(2,l)$ and Chern Classes of Heegner Divisors},
    author = {J. H. Bruinier},
    series = {Lecture Notes in Mathematics},
    year = {1997},
    publisher = {Springer Berlin, Heidelberg},
    pages="VIII, 156"
}

@book{Courtieu1991,
author="Courtieu, M.
and Panchishkin, A. A.",
title="Non-Archimedean L-Functions and Arithmetical Siegel Modular Forms: Second, Augmented Edition",
year="1991",
publisher="Springer Berlin Heidelberg",
series="Lecture Notes in Mathematics",
pages="VIII, 204"
}

@article{cayley_identity,
title = {Algebraic/combinatorial proofs of Cayley-type identities for derivatives of determinants and pfaffians},
journal = {Advances in Applied Mathematics},
volume = {\textbf{50}},
number = {\textbf{4}},
pages = {474-594},
year = {2013},
author = {S. Caracciolo and A. D. Sokal and A. Sportiello}
}

@article{deitmar_krieg,
year = {1991},
volume = {\textbf{208}},
pages = {273–288},
author = {A. Deitmar and A. Krieg},
title = {{Theta correspondence for Eisenstein series.}},
journal = {Math Z},
}

@article{gritsenko_jacobi,
  title={{Fourier-Jacobi functions of n.}},
  author={Gritsenko, V. A.},
  journal={J Math Sci},
  volume={\textbf{53}},
  pages={243-252},
  year={1991}
}

@article{gritsenko,
    author = "Gritsenko, V. A.",
    title = "Jacobi functions and Euler products for Hermitian modular forms.",
    journal = "J Math Sci",
    volume = "\textbf{62}",
    pages = {2883-2914},
    year = "1992"
}

@phdthesis{hauffe21,
author       = {Hauffe-Waschbüsch, A.},
title        = {{V}erschiedene {A}spekte von {M}odulformen in mehreren {V}ariablen},
school       = {RWTH Aachen University},
type         = {PhD Thesis},
address      = {Aachen},
publisher    = {RWTH Aachen University},
year         = {2021},
url         = {https://publications.rwth-aachen.de/record/824384/files/824384.pdf},
}

@unpublished{hanke2011enumerating,
      title={{Enumerating maximal definite quadratic forms of bounded class number over $\mathbb{Z}$ in $n \geq 3$ variables}}, 
      author={J. Hanke},
      year={2011},
      primaryClass={math.NT},
      url={https://arxiv.org/abs/1110.1876},
      note={Preprint}   
}

@article{kalinin,
year = {1977},
publisher = {},
volume = {\textbf{32}},
number = {\textbf{4}},
pages = {449--476},
author = {V. L. Kalinin},
title = {Eisenstein series on the symplectic group},
journal = {Mathematics of the USSR-Sbornik}
}

@incollection{krieg_integral_orthogonal,
  title={Integral orthogonal groups},
  author={Krieg, A.},
  booktitle={Dynamical Systems, Number Theory and Applications: A Festschrift in Honor of Armin Leutbecher's 80th Birthday},
  pages={177--195},
  year={2016},
  publisher={World Scientific}
}

@book {Krieg_Book,
    AUTHOR = {Krieg, A.},
     TITLE = {Modular forms on half-spaces of quaternions},
    SERIES = {Lecture Notes in Mathematics},
    VOLUME = {\textbf{1143}},
 PUBLISHER = {Springer-Verlag, Berlin},
      YEAR = {1985},
     PAGES = {xiii+203},
      ISBN = {3-540-15679-8},
   MRCLASS = {11F55 (11F46 11H55)},
  MRNUMBER = {807947},
MRREVIEWER = {K.-B.\ Gundlach},
}

@article{krieg1991,
author = {Krieg, A.},
journal = {Acta Arithmetica},
number = {\textbf{3}},
pages = {243-259},
title = {A Dirichlet series for modular forms of degree n},
volume = {\textbf{59}},
year = {1991},
}

@article{krieg_jacobi,
title = {{Jacobi Forms of Several Variables and the Maaß Space}},
journal = {Journal of Number Theory},
volume = {\textbf{56}},
number = {\textbf{2}},
pages = {242-255},
year = {1996},
author = {A. Krieg},
}

@article{kohnen_skoruppa,
    author = {Kohnen, W. and Skoruppa, N. P.},
    title = {{A certain Dirichlet series attached to Siegel modular forms of degree two.}},
    journal = {Invent Math},
    volume = {\textbf{95}},
    pages = {541–558},
    year = {1989},
}

@book{maass,
    title = {Siegel's Modular Forms and Dirichlet Series},
    author = {H. Maaß},
    series = {Lecture Notes in Mathematics},
    year = {1971},
    publisher = {Springer Berlin, Heidelberg}
}

@phdthesis{jacobi_lattice,
           title = {{On Jacobi forms of lattice index}},
          author = {A. Mocanu},
          school = {University of Nottingham},
            year = {2019},
         url = {https://eprints.nottingham.ac.uk/56893/1/thesis.pdf},
}

@article{psyroukis_orthogonal,
      title={A Fourier-Jacobi Dirichlet series for cusp forms on orthogonal groups}, 
      author={R. Psyroukis},
      year={2025},
      journal = {Research in Number Theory},
      volume = {\textbf{90}},
      number = {\textbf{11}},
}

@article{raghavan,
author = {Raghavan, S. and Sengupta, J.},
journal = {Acta Arithmetica},
number = {\textbf{2}},
pages = {181-201},
title = {A Dirichlet series for Hermitian modular forms of degree 2},
volume = {\textbf{58}},
year = {1991},
}

@article{satoh1986differential,
  title={Differential operators and congruences for Siegel modular forms of degree two},
  author={Satoh, T.},
  journal={Journal of the Mathematical Society of Japan},
  volume={\textbf{38}},
  number={\textbf{1}},
  pages={127-146},
  year={1986},
  publisher={The Mathematical Society of Japan}
}

@phdthesis{eisenstein_thesis,
      author       = {Schaps, F.},
      othercontributors = {Krieg, Aloys and Heim, Bernhard and Alfes-Neumann, Claudia},
      title        = {{{E}isenstein series for the orthogonal group $\textup{O}(2,n)$}},
      school       = {RWTH Aachen University},
      type         = {Dissertation},
      publisher    = {RWTH Aachen University},
      reportid     = {RWTH-2022-08607},
      year         = {2022},
      cin          = {114110 / 110000},
      ddc          = {510},
      cid          = {$I:(DE-82)114110_20140620$ / $I:(DE-82)110000_20140620$},
      typ          = {PUB:(DE-HGF)11},
      url = {https://publications.rwth-aachen.de/record/853111/files/853111.pdf}
}

@book{shimura2004arithmetic,
  title={Arithmetic and Analytic Theories of Quadratic Forms and Clifford Groups},
  author={Shimura, G.},
  volume = {\textbf{109}},
  series={Mathematical Surveys and Monographs},
  year={2004},
  publisher={American Mathematical Society}
}

@article{shimura_differential,
 author = {G. Shimura},
 journal = {Annals of Mathematics},
 number = {\textbf{2}},
 pages = {237--272},
 publisher = {[Annals of Mathematics, Trustees of Princeton University on Behalf of the Annals of Mathematics, Mathematics Department, Princeton University]},
 title = {Invariant Differential Operators on Hermitian Symmetric Spaces},
 volume = {\textbf{132}},
 year = {1990}
}

@book{Shimura_Euler_Product,
    title = {Euler Products and Eisenstein Series},
    author = {Shimura, G.},
    series = {CBMS Regional Conference Series in Mathematics},
    year = {1997},
    publisher = {American Mathematical Society},
    volume = {\textbf{93}},
}

@article{sugano,
author = {Sugano, T.},
title = {{On Dirichlet Series Attached to Holomorphic Cusp Forms on $\textup{SO}(2,q)$}},
volume = {\textbf{62}},
journal = {Adv. Stud. Pure Math.},
pages = {333-362},
year = {1985},
}

@inbook{yang,
author="Yang, J-H",
title="Geometry and Arithmetic on the Siegel--Jacobi Space",
bookTitle="Geometry and Analysis on Manifolds: In Memory of Professor Shoshichi Kobayashi",
year="2015",
publisher="Springer International Publishing",
pages="275--325",
}

@article{yamazaki, 
title={Rankin-Selberg method for Siegel cusp forms}, 
volume={\textbf{120}},
journal={Nagoya Mathematical Journal}, 
author={Yamazaki, T.}, 
year={1990}, 
pages={35–49},
}

@article{richter,
author = {Richter, Olav},
year = {2004},
month = {12},
pages = {},
title = {On Transformation Laws for Theta Functions},
volume = {34},
journal = {Rocky Mountain Journal of Mathematics},
doi = {10.1216/rmjm/1181069809}
}

@article{one-dimensional-cusps,
  title={On the number of cusps of orthogonal Shimura varieties},
  author={Attwell-Duval, D.},
  journal={Annales math{\'e}matiques du Qu{\'e}bec},
  volume={\textbf{38}},
  number={\textbf{2}},
  pages={119--131},
  year={2014},
  publisher={Springer}
}

@article{single-class-genera,
  title={Single-class genera of positive integral lattices},
  author={Lorch, D. and Kirschmer, M.},
  journal={LMS Journal of Computation and Mathematics},
  volume={\textbf{16}},
  pages={172--186},
  year={2013},
  publisher={London Mathematical Society}
}
\end{document}